\documentclass[reqno,english]{amsart}
\usepackage{latexsym,verbatim,amscd,mathrsfs,array}
\usepackage{amsmath,amssymb,amsthm,amsfonts,graphicx,color}

\usepackage{pdfsync}
\usepackage{epstopdf}

\usepackage{cite}

\newcommand{\C}{{\mathbb C}}
\newcommand{\R} {\mathbb R}

\newcommand{\normadora}{\ltimes}

\newtheorem{theorem}{Theorem}[section]
\newtheorem{lemma}[theorem]{Lemma}

\newtheorem{prop}[theorem]{Proposition}

\numberwithin{equation}{section}

\title[Multispike solutions for the Brezis-Nirenberg problem in 3D]{Multispike solutions for the Brezis-Nirenberg problem in dimension three}

\author{Monica Musso}
\address{M. Musso - Departamento de Matem\'atica, Pontificia Universidad Cat\'olica de Chile, Avda. Vicu\~na Mackenna 4860, Macul, Chile.}
\email{mmusso@mat.puc.cl}

\author{Dora Salazar}
\address{D. Salazar -
Escuela de Matem\'aticas, Universidad Nacional de Colombia Sede Medell\'in, Apartado A\'ereo 3840, Medell\'in, Colombia.}
\email{dcsalazarl@unal.edu.co}

\begin{document}

\begin{abstract}
We consider the problem  $\Delta u + \lambda u +u^5 = 0$, $u>0$, in a smooth bounded domain $\Omega$ in $\R^3$, under zero Dirichlet boundary conditions.
We obtain solutions to this problem exhibiting multiple bubbling behavior at $k$ different points of the domain as $\lambda$ tends to a special positive value $\lambda_0$, which we characterize in terms of the Green function of $-\Delta - \lambda$.
\end{abstract}

\maketitle

\section{Introduction}
\noindent
Let us consider the Brezis-Nirenberg  problem
\begin{align*}
(\wp_\lambda)\,
\left\{
\begin{aligned}
& \Delta u +\lambda u + u^p  =0
\quad\text{in } \Omega,
\\
& u>0 \quad\text{in } \Omega,
\\
& u=0 \quad\text{on } \partial\Omega,
\end{aligned}
\right.
\end{align*}
where $\Omega$  is a smooth bounded domain in $\R^N$, $N\geq 3$, $p= \frac{N+2}{N-2}$ and
 $\lambda$ is a real positive parameter.

In this article, we are interested in obtaining solutions to this problem, in the special case $N=3$, that concentrate in $k$ different points of $\Omega$, $k\geq 2$.
In particular, we
analyze the role of the Green function of $\Delta+\lambda$ in the presence of multi-peak solutions when $\lambda$ is regarded as a parameter.

Solutions to $(\wp_\lambda)$ correspond to critical points of the energy functional
\[
J_\lambda(u)=\frac{1}{2}\int_{\Omega}\vert\nabla u\vert^2-\frac{\lambda}{2}\int_{\Omega}u^2-\frac{1}{p+1}\int_{\Omega}|u|^{p+1}.
\]
Although this functional is of class $C^2$ in $H_0^1(\Omega)$, it does not satisfy the Palais-Smale condition at all energy levels, and hence variational arguments to find solutions are delicate and sometimes fail.

Let $\lambda_1$ denote the first eigenvalue of $-\Delta$ with Dirichlet boundary condition.
It is well known that
$(\wp_\lambda)$ admits no solutions if
$\lambda \geq\lambda_1$, which can be verified by
testing the equation against a first eigenfunction of the Laplacian. Moreover, the classical Pohozaev identity \cite{pohozaev} guarantees that problem $(\wp_\lambda)$ with $\lambda\leq 0$ has no solution in a starshaped domain.

In the classical paper \cite{brezis-nirenberg}, Brezis and Nirenberg showed that
least energy
solutions
to this problem exist for $\lambda \in (\lambda^*,\lambda_1)$, where $\lambda^* \in [0,\lambda_1)$ is a special number depending on the domain.
They also showed that if $N\geq 4$, then $\lambda^*=0$ and in particular $(\wp_\lambda)$ has a solution with minimal energy for all $\lambda \in (0,\lambda_1)$.

When $N=3$ the situation is strikingly different, since, as it is shown in \cite{brezis-nirenberg}, $\lambda^*>0$ and no solutions with minimal energy exist
when  $\lambda \in (0,\lambda^*)$.
In 2002, Druet \cite{druet} showed that there is no solution with minimal energy neither for $\lambda=\lambda^*$,
which implies that $\lambda^*$ can be characterized as the critical value such that a solution of $(\wp_\lambda)$ with minimal energy exists if and only if $\lambda\in (\lambda^*,\lambda_1)$.

In the particular case of the ball in $\R^3$, Brezis and Nirenberg \cite{brezis-nirenberg} also proved that $\lambda^* = \frac{\lambda_1}{4}$ and that a solution to $(\wp_\lambda)$ exists if and only if $\lambda \in (  \frac{\lambda_1}{4},  \lambda_1)$. By the results of Gidas, Ni, Nirenberg \cite{gidas-ni-nirenberg} and  Adimurthi, Yadava \cite{adimurthi-yadava} this solution is unique and corresponds indeed to the minimum of the energy functional.

In dimensions three  a characterization of $\lambda^*$ can be given in terms of the \emph{Robin function $g_\lambda$} defined as follows.
Let $\lambda\in(0,\lambda_1)$. For a given $x\in\Omega$ consider the  Green function $G_\lambda(x, y)$, solution of
\[
\begin{array}{rlll}
-\Delta_yG_\lambda-\lambda G_\lambda&=&\delta_x&y\in \Omega,
\\
G_\lambda(x, y)&=&0&y\in \partial\Omega,
\end{array}
\]
where $\delta_x$ is the Dirac delta at $x$.
Let $H_{\lambda}(x,y) = \Gamma(y-x)-G_\lambda(x,y)$ with $\Gamma(z)=\frac{1}{4\pi \vert z\vert}$, be its regular part,
and let us define the Robin function of $G_\lambda$ as
$g_\lambda(x) := H_\lambda(x, x)$.

It is known that $g_\lambda(x)$ is a smooth function which goes to $+\infty$ as $x$ approaches $\partial \Omega$. The minimum of $g_\lambda$ in $\Omega$ is strictly decreasing in $\lambda$, is strictly positive when $\lambda$ is close to $0$ and approaches $-\infty$ as $\lambda\uparrow \lambda_1$.

It was conjectured in \cite{brezis}  and proved by Druet \cite{druet} that $\lambda^*$ is the largest $\lambda \in (0,\lambda_1)$ such that $\min_{\Omega} g_\lambda \geq 0$. Moreover, Druet also proved that, as $\lambda\downarrow \lambda^*$, least energy solutions to $(\wp_\lambda)$ develop a singularity which is located at a point $\zeta_0\in \Omega$ such that $g_{\lambda^*}(\zeta_0) = 0$.
Note that $\zeta_0$ is  a global minimizer of  $g_{\lambda^*}$ and hence a critical point.
A concentrating family of solutions can exist at other values of  $\lambda$.
Indeed,   del Pino, Dolbeault and Musso \cite{DPDM}
proved that if $\lambda_0 \in (0,\lambda_1)$ and $\zeta_0\in\Omega$  are such that
\[
g_{\lambda_0}(\zeta_0)=0, \quad
\nabla g_{\lambda_0}(\zeta_0)=0,
\]
and either $\zeta^0$ is a strict local minimum or a nondegenerate critical point of $g_{\lambda}$, then for $\lambda - \lambda_0 >0$,  there is a solution $u_\lambda$ of $(\wp_\lambda)$ such that
\[
u_\lambda (x) = w_{\mu,\zeta}\,(1+o(1))
\]
in $\Omega$ as $\lambda - \lambda_0 \to 0$, where
\[
w_{\mu,\zeta}(x) =
\frac{\alpha_3 \, \mu^{1/2}}{(\mu^2+ | x-\zeta|^2 )^{1/2}},
\quad \alpha_3 = 3^{1/4},
\]
$\zeta\to \zeta_0$ and $\mu=O(\lambda - \lambda_0)$.

The behavior described above, namely \emph{bubbling} of a family of solutions, was
already studied in higher dimensions.
Han \cite{han} proved  that if $N\geq 4 $,  minimal energy solution of $(\wp_\lambda)$ concentrate at a critical point of the Robin function $g_0$ as $\lambda\downarrow0$. See also Rey \cite{Rey}  for an arbitrary family of solutions that concentrates at a single point. Conversely, Rey in \cite{Rey,Rey2} showed that attached to any $C^1$-stable critical point of the Robin function $g_0$ there is a family of solutions of $(\wp_\lambda)$ that blows up at this point as $\lambda\downarrow 0$.

Unlike  the case of dimension three, bubbling behavior  with concentration at multiple points as $\lambda \downarrow 0$ is known in higher dimensions.
Indeed, Musso and Pistoia \cite{Musso-Pistoia2002} constructed  multispike solutions in a smooth bounded domain $\Omega \subset \R^N$, $N\geq 5$.
To state precisely their result let us consider an integer $k \geq 1$,
let us write $\bar\mu = (\bar\mu_1,\ldots,\bar\mu_k)\in \R^k $, $\zeta = (\zeta_1,\ldots,\zeta_k) \in \Omega^k$, $\zeta_i\not=\zeta_j$ for $i\not=j$, and define
\[
\psi_k(\bar\mu,\zeta) = \frac{1}{2} ( M(\zeta) \,\bar\mu^{\frac{N-2}{2}} , \bar\mu^{\frac{N-2}{2}} )  - \frac{1}{2}
\,B\sum_{i=1}^k\bar\mu_i^2
\]
where $\bar\mu^{\frac{N-2}{2}} = (\bar\mu_1^{\frac{N-2}{2}},\ldots,\bar\mu_k^{\frac{N-2}{2}})$, and $M(\zeta)$ is the matrix with coefficients
\[
m_{ii}(\zeta) = g_0(\zeta_i), \quad
m_{ij}(\zeta) = -G_0(\zeta_i,\zeta_j), \quad \text{for } i\not=j.
\]
Here $B>0$ is a constant depending only on the dimension.
It is shown in  \cite{Musso-Pistoia2002}  that if $\psi_k$ has a stable critical point $(\bar\mu,\zeta)$ then, for $\lambda>0$ small, problem $(\wp_\lambda)$ has a family of solutions that blow up at the $k$ points $\zeta_1,\ldots,\zeta_k$, with profile near $\zeta_i$ given by  $w_{\mu_i,\zeta_i}$ and rates $\mu_i \sim \bar\mu_i \,\lambda^{\frac{1}{N-4}}$. Musso and Pistoia also exhibit classes of domains where such critical points of $\psi_k$ can be found.
A related multiplicity result is given by the same authors in \cite{Musso-Pistoia2003}, where $\Omega$ is a domain with a sufficiently small hole.
They show that for $\lambda<0$ small there is a family of solutions concentrating at two points.


As far as we know, there are no works dealing with solutions with multiple concentration in lower dimensions ($N=3$ and $N=4$), and it is not clear what type of finite dimensional function governs the location and the  concentration rate of the bubbling solutions.

In this work we focus in dimension three. We give conditions on the parameter $\lambda$ such that solutions with simultaneous concentration at $k$ points exist and find the finite dimensional function describing the location and rate of concentration.
We remark that the condition on $\lambda$ that we obtain for solutions with multiple bubbling in dimension three is a non-obvious but natural generalization of the condition given by  Dolbeault, del Pino and Musso  \cite{{DPDM}} for single bubble solutions in dimension three,
and is  somehow related to the result of Musso and Pistoia \cite{Musso-Pistoia2002} for $\lambda^*=0$ in higher dimensions.

In order to state our results we need some notation. For a given integer
$k\geq2$
set
\[
\Omega_k^* =\{ \zeta= ( \zeta_1,\dots, \zeta_k) \in \Omega^k : \zeta_i\not=\zeta_j \text{ for all } i\not=j\}.
\]
For $\zeta= ( \zeta_1,\dots, \zeta_k) \in \Omega_k^*$, let us consider the matrix
\[
M_\lambda(\zeta):=
\begin{pmatrix}
g_{\lambda}(\zeta_1) & -G_{\lambda}(\zeta_1,\zeta_2) & \ldots &  -G_{\lambda}(\zeta_1,\zeta_k) \\
-G_{\lambda}(\zeta_1,\zeta_2) & g_{\lambda}(\zeta_2) & \ldots & -G_\lambda(\zeta_2,\zeta_k)
\\
\vdots & & & \vdots
\\
-G_\lambda(\zeta_1,\zeta_k) & -G_\lambda(\zeta_2,\zeta_k) & \ldots & g_\lambda(\zeta_k)
\end{pmatrix}.
\]
In other words,  $M_\lambda(\zeta)$ is the matrix whose
$ij$ component is given by
\[
\begin{cases}
g_\lambda(\zeta_i) & \text{if } i=j
\\
- G_\lambda(\zeta_i,\zeta_j) & \text{if } i\not=j .
\end{cases}
\]
Define the function
\begin{align*}
\psi_\lambda(\zeta) = \det M_\lambda(\zeta)  , \quad \zeta \in \Omega_k^*.
\end{align*}
Our main result  is the following.
\begin{theorem}
\label{thm1}
Assume that for a number $\lambda=\lambda_0 \in (0,\lambda_1)$ there is \,$\zeta^0 = ( \zeta_1^0,\ldots,\zeta_k^0) \in\Omega_k^*$ such that:
\begin{itemize}

\item[(i)]
$\psi_{\lambda_0}(\zeta^0)=0$  and $M_{\lambda_0}(\zeta^0)$ is positive semidefinite,

\item[(ii)]
$ D_{\zeta}\psi_{\lambda_0}(\zeta^0)=0 $,

\item[(iii)]
$D_{\zeta\zeta} ^2\psi_{\lambda_0}(\zeta^0)$ is non-singular,

\item[(iv)]
$\frac{\partial  \psi_{\lambda} }{\partial \lambda}\big|_{\lambda=\lambda_0} (\zeta^0) < 0 $.

\end{itemize}
Then for $\lambda=\lambda_0+\varepsilon$, with $\varepsilon>0$ small, problem $(\wp_\lambda)$ has a solution $u$ of the form
\[
u = \sum_{j=1}^k w_{\mu_j,\zeta_j} +  O(\varepsilon^{\frac{1}{2}})
\]
where   $\mu_j = O(\varepsilon)$, $\zeta_j\to \zeta_j^0$, $j=1,\ldots,k$, and $O(\varepsilon^{\frac{1}{2}})$ is uniform in $\Omega$ as $\varepsilon\to 0$.
\end{theorem}

\medskip	

\noindent
We remark that
Theorem~\ref{thm1} admits some variants. For example, if $\frac{\partial  \psi_{\lambda} }{\partial \lambda}\big|_{\lambda=\lambda_0} (\zeta^0)  > 0 $, then a solution with $k$ bubbles can be found for  $\lambda=\lambda_0-\varepsilon$, with $\varepsilon>0$ small.
When $k=2$ the assumption that $M_{\lambda_0}(\zeta^0)$ is positive semidefinite is equivalent to $g_{\lambda_0}(\zeta_1^0)>0$ or  $g_{\lambda_0}(\zeta_2^0)>0$.

As an example where the previous theorem can be applied,
let us consider the annulus
\[
\Omega_a  = \{ x\in \R^3 \ : \  a < |x| < 1 \} ,
\]
where $0<a<1$. From the work of Kazdan and Warner \cite{kazdan-warner} it is known that for any $\lambda<\lambda_1$ there is a radial positive solution in $\Omega_a$.

For each $k\geq 2$, we prove that there exists $0<a_k<1$ such that if $a \in (a_k,1)$, then problem $(\wp_{\lambda_0+\varepsilon})$  in $\Omega_a$, $\varepsilon>0$ small, has a solution with $k$ bubbles centered at the vertices of a planar regular polygon for some $\lambda_0\in (0,\lambda_1)$. As a byproduct of the construction we also deduce that
\[
\lambda_0<\lambda^*.
\]
A detailed proof of these assertions is given in Section~\ref{exampleAnnulus}.
The ideas developed here can be applied
to obtain two bubble solutions in more general thin axially symmetric domains.

In dimension $N\geq 4$ qualitative similar solutions were detected by Wang-Willem \cite{wang-willem} for all $\lambda $ in an interval almost equal to $ (0,\lambda_1)$ by using variational methods.
The existence of this kind of solutions in dimension three was (to the best of our knowledge) not known.

We should remark that  multipeak solutions cannot be constructed in a ball, since the solution of $(\wp_\lambda)$ is radial and unique if it exists. This may indicate that if we consider  $(\wp_\lambda)$ in the annulus $\Omega_a$ with $a>0$ sufficiently small there are no multipeak solutions.

Finally, we mention that several interesting results have been obtained on the existence of sign changing solutions to the  Brezis-Nirenberg problem. See for instance Ben Ayed, El Mehdi, Pacella \cite{benayed-elmehdi-pacella},  Iacopetti \cite{iacopetti}, Iacopetti and Vaira \cite{iacopetti-vaira} and the references therein. It is in fact foreseeable that the methods developed in this work can also give the existence of multipeak sign changing solutions in dimension 3.

The paper is organized as follows. In Section~\ref{sectEnergyExpansion} we introduce some notation and give the energy expansion for a multi-bubble approximation.
Sections~\ref{sectLinear} and ~\ref{sectNonlinear} are respectively devoted to the study the linear and nonlinear problems involved in the Lyapunov Schmitd reduction, which is carried out in Section~\ref{secReduction}.
Theorem~\ref{thm1} is proved in Section~\ref{sectProof}.
Finally, in Section~\ref{exampleAnnulus} we give the details for the case of the annulus $\Omega_a$.

\section{Energy expansion of a multi-bubble approximation}
\label{sectEnergyExpansion}

\noindent
We denote by
\[
U(z):=\frac{\alpha_3 }{(1+\vert z\vert^2)^{1/2}},\quad \alpha_3=3^{1/4},
\]
the standard bubble.
It is well known that all positive solutions to the Yamabe equation
\[
\Delta w + w^5 = 0 \quad\text{in }\mathbb{R}^3
\]
are of the form
\begin{align*}
w_{\mu,\zeta}(x):
&=\mu^{-1/2}\,U\Bigl(\frac{x-\zeta}{\mu}
\Bigr)
=\frac{\alpha_3 \, \mu^{1/2}}{\Bigl(\mu^2+\vert x-\zeta\vert^2\Bigr)^{1/2}},
\end{align*}
where $\zeta$ is a point in $\mathbb{R}^3$ and $\mu$ is a positive number.

From now on we assume that $0<\lambda<\lambda_1(\Omega)$.

For a given $k\geq 2$,
we consider $k$ different points $\zeta_1,\ldots,\zeta_k\in\Omega$ and
small positive numbers $\mu_1,\ldots, \mu_k$ and denote by
\[
w_i:=w_{\mu_i,\zeta_i}.
\]
We are looking for solutions of $(\wp_\lambda)$ that at main order are given by $\sum_{i=1}^k w_i$. Since $w_i$ are not zero on $\partial\Omega$ it is natural to correct this approximation by terms that provide the Dirichlet boundary condition.
In order to do this we introduce, for each $i=1,\ldots,k$, the function
$\pi_i$ defined as  the unique solution of the problem
\[
\begin{array}{rlll}
\Delta \pi_i+\lambda\,\pi_i&=&-\lambda\,w_i&\text{in }\Omega,
\\
\pi_i&=&-w_i&\text{on }\partial\Omega,
\end{array}
\]
and then we  shall consider as a first approximation of the solution to $(\wp_\lambda)$ one of the form
\[
U^0 = U_1+\ldots+U_k,
\]
where
\[
U_i(x) = w_i(x) + \pi_i(x).
\]
Observe that $U_i\in H_0^1(\Omega)$ and satisfies the equation
\begin{equation}
\label{projection}
\left\{
\begin{array}{rlll}
\Delta U_i+\lambda U_i&=&-w_i^5&\text{in } \Omega,
\\
U_i&=&0&\text{on }\partial\Omega.
\end{array}
\right.
\end{equation}
Let us recall that the energy functional associated to $(\wp_\lambda)$ when $N=3$ is given by:
\[
J_\lambda(u)=\frac{1}{2}\int_{\Omega}\vert\nabla u\vert^2-\frac{\lambda}{2}\int_{\Omega}u^2-\frac{1}{6}\int_{\Omega}|u|^{6}.
\]
Let us write $\zeta= (\zeta_1,\ldots,\zeta_k)$ and $\mu=(\mu_1,\ldots,\mu_k)$ and note that $U^0 = U^0(\mu,\zeta)$.
Since we are looking for solutions close to $U^0(\mu,\zeta)$,  formally we expect
$
J_\lambda( U^0(\mu,\zeta) )$
to be almost critical in the parameters $\mu,\zeta$.
For this reason it is important to obtain an asymptotic formula of the functional $(\mu,\zeta)\rightarrow J_\lambda( U^0(\mu,\zeta))$ as $\mu\to 0$.

For any $\delta>0$ set
\begin{multline*}
\Omega_\delta^k:=\{
\zeta\equiv(\zeta_1,\ldots,\zeta_k)\in\Omega^k: \,\textrm{dist}(\zeta_i,\partial \Omega)>\delta, \vert \zeta_i-\zeta_j\vert>\delta,\\
 i=1,\ldots,k,\,j=1,\ldots,k,\, i\neq j
\} .
\end{multline*}
The main result in this section is the expansion of the energy in the case of a multi-bubble ansatz.

\begin{lemma}
\label{lemmaEnergyExpansion}
Let $\delta>0$ be fixed and let $\zeta\in\Omega_\delta^k$.
Then as $\mu_i\rightarrow 0$, the following expansion holds:
\begin{align*}
J_\lambda\Bigl(\sum_{i=1}^k U_i\Bigr):=&\,k\,a_0
+a_1\sum_{i=1}^k\Bigl(\mu_i\,g_{\lambda}(\zeta_i)-\sum_{j\neq i}\mu_i^{1/2}\,\mu_j^{1/2}\,G_{\lambda}(\zeta_i,\zeta_j)\Bigr)+a_2\,\lambda\,\sum_{i=1}^k \mu_i^2
\\
&
-a_3\,\sum_{i=1}^k\Bigl(\mu_i\,g_{\lambda}(\zeta_i)-\sum_{j\neq i}\mu_i^{1/2}\mu_j^{1/2}\,G_{\lambda}(\zeta_i,\zeta_j)\Bigr)^2
+ \theta_\lambda^{(1)} (\mu,\zeta) ,
\end{align*}
where  $\theta_\lambda^{(1)} (\zeta,\mu)$ is such  that for any $\sigma>0$ and $\delta>0$ there is $C$ such that
\[
\Bigl| \frac{\partial^{m+n}}{\partial\zeta^m\partial\mu^n}\,\theta_\lambda^{(1)}  (\zeta,\mu) \Bigr| \leq C (\mu_1+\ldots+\mu_k)^{3-\sigma -n} ,
%
\]
for   $m = 0,1$, $n = 0,1,2$,  $m+n\leq 2$, all small $\mu_i$, $i=1,2,\ldots,k$, and all $\zeta\in\Omega_\delta^k$.
\end{lemma}

The $a_j$'s are
the following explicit constants
\begin{align}
\label{a0}
a_0:&=\frac{1}{3}\int_{\R^3}U^6
= \frac{1}{4}(\alpha_3\pi)^2,
\\
\label{a1}
a_1:&=2\pi\alpha_3\int_{\R^3}U^5
=8(\alpha_3\pi)^2,
\\
\label{a2}
a_2:&=\frac{\alpha_3}{2}\int_{\R^3}\biggl[\biggl(\frac{1}{\vert z\vert}-\frac{1}{\sqrt{1+\vert z\vert^2}}
\biggr)U+\frac{1}{2}\,\vert z\vert\, U^5\biggr]\,dz
a_2=(\alpha_3\pi)^2,
\\
\label{a3}
a_3:&=\frac{5}{2}(4\pi\alpha_3)^2\int_{\R^3}U^4
=120\,(\alpha_3\pi^2)^2.
\end{align}
To prove this lemma we need some preliminary results.
To begin with, we recall the relationship between the functions $\pi_{i}(x)$ and the regular part of Green's function, $H_\lambda(\zeta_i,x)$.
Let us consider the (unique) radial solution $\mathcal{D}_0(z)$ of the following problem in entire space
\[
\begin{array}{rlll}
\Delta \mathcal{D}_0&=&-\lambda\,\alpha_3\,
\Bigl(
\frac{1}{(1+\vert z\vert^2)^{1/2}}
-\frac{1}{\vert z\vert}
\Bigr)
&\text{in } \mathbb{R}^3,
\\
\mathcal{D}_0&\rightarrow
&0&
\text{as } \vert z\vert\rightarrow \infty.
\end{array}
\]
Then $\mathcal{D}_0(z)$ is a $C^{0,1}$
function with $\mathcal{D}_0(z)\sim \vert z\vert^{-1}\log \vert z\vert$ as $\vert z\vert\rightarrow \infty$.

\bigskip

\begin{lemma}
\label{lemma22}
For any $\sigma > 0$ the following expansion holds as $\mu_i\rightarrow 0$
\[
\mu_i^{-\frac{1}{2}}\pi_i(x)=-4\pi\alpha_3\,H_{\lambda}(x,\zeta_i)+\mu_i\,\mathcal{D}_0\Bigl(\frac{x-\zeta_i}{\mu_i}\Bigr)+\mu_i^{2-\sigma}\,\theta(\mu_i,x,\zeta_i)
\]
where for $m=0,1$, $n=0,1,2$, $m+n\leq 2$, the function
$\mu_i^n\frac{\partial^{m+n}}{\partial\zeta_i^m\partial\mu_i^n}\,\theta(\mu_i,y,\zeta_i)
$
is bounded uniformly on $y\in\Omega$, all small $\mu_i$ and $\zeta_i$ in compact subsets of $\Omega$.
\end{lemma}
\begin{proof}
See \cite[Lemma 2.2]{DPDM}.
\end{proof}


From Lemma \ref{lemma22} and the fact that,
away from $ x=\zeta_i$,
\[
\mathcal{D}_0\Bigl(\frac{x-\zeta_i}{\mu_i}\Bigr)=O(\mu_i\log \mu_i),
\]
the following holds true.
\begin{lemma}
\label{Uiexp}
Let $\delta>0$ be given. Then for any $\sigma > 0$ and $x\in\Omega\setminus B_\delta(\zeta_i)$ the following expansion holds as $\mu_i\rightarrow 0$
\[
\mu_i^{-\frac{1}{2}}U_i(x)=4\pi\,\alpha_3\,G_{\lambda}(x,\zeta_i)+\mu_i^{2-\sigma}\,\hat{\theta}(\mu_i,x,\zeta_i)
\]
where for $m=0,1$, $n=0,1,2$, $m+n\leq 2$, the function
$\mu_i^n\frac{\partial^{m+n}}{\partial\zeta_i^m\partial\mu_i^n}\,\hat{\theta}(\mu_i,x,\zeta_i)
$
is bounded uniformly on $x\in\Omega\setminus B_\delta(\zeta_i)$, all small $\mu_i$ and $\zeta_i$ in compact subsets of $\Omega$.
\end{lemma}

We also recall the expansion of the energy for the case of a single bubble, which was proved in \cite{DPDM}.

\begin{lemma}
\label{lemma21}
For any $\sigma > 0$ the following expansion holds as $\mu_i\rightarrow 0$
\begin{equation*}
J_\lambda(U_i)=a_0+a_1\,g_{\lambda}(\zeta_i)\,\mu_i+
\bigl(a_2\,\lambda-a_3\,g_{\lambda}(\zeta_i)^2\bigr)\,\mu_i^2+\mu_i^{3-\sigma}\,
\theta(\mu_i,\zeta_i),
\end{equation*}
where for $m=0,1$, $n=0,1,2$, $m+n\leq 2$, the function
$\mu_i^n\frac{\partial^{m+n}}{\partial\zeta_i^m\partial\mu_i^n}\,\theta(\mu_i,\zeta_i)
$
is bounded uniformly on all small $\mu_i$ and $\zeta_i$ in compact subsets of $\Omega$.
The $a_j$'s are given in  \eqref{a0}--\eqref{a3}.
\end{lemma}

\bigskip

\begin{proof}[Proof of Lemma~\ref{lemmaEnergyExpansion}]
\noindent
We decompose
\begin{align*}
J_\lambda \Bigl(\sum_{i=1}^k U_i\Bigr)
&=\frac{1}{2}\sum_{i=1}^k\Bigl(\int_{\Omega} \vert\nabla U_i\vert^2+\sum_{j\neq i}\int_{\Omega}\nabla U_i\cdot \nabla U_j\Bigr)
\\
&\quad -\frac{\lambda}{2}\sum_{i=1}^k\Bigl( \int_{\Omega}U_i^2+\sum_{j\neq i}\int_{\Omega}U_i\,U_j\Bigr)-\frac{1}{6}\int_{\Omega}\Bigl(\sum_{i=1}^k U_i\Bigr)^6
\\
&  =\sum_{i=1}^kJ_\lambda(U_i)+\frac{1}{2}\sum_{i=1}^k\sum_{j\neq i}\int_{\Omega}[\nabla U_i\cdot \nabla U_j-\lambda \,U_i\,U_j]
\\
& \quad -\frac{1}{6}\int_{\Omega}\Bigl[\Bigl(\sum_{i=1}^k U_i\Bigr)^6-\sum_{i=1}^k U_i^6\Bigr].
\end{align*}
Integrating by parts in $\Omega$ we get
\[
\int_{\Omega}\nabla U_i\cdot \nabla U_j=
\int_{\Omega}(-\Delta U_i)U_j+
\int_{\partial\Omega}\frac{\partial U_i}{\partial \eta} U_j=\int_{\Omega}(-\Delta U_i)U_j,
\]
where $\frac{\partial }{\partial \eta}$ denotes the derivative along the unit outgoing normal at a point of $\partial \Omega$.
From \eqref{projection} one gets
\[
\int_{\Omega}\nabla U_i\cdot \nabla U_j=\int_{\Omega}(-\Delta U_i)U_j=\int_{\Omega}(\lambda U_i+w_i^5)U_j.
\]
and so
\[
\int_{\Omega}\nabla U_i\cdot \nabla U_j-\lambda\int_{\Omega}U_i\, U_j=\int_{\Omega}w_i^5\, U_j.
\]
Hence,
\begin{equation}
J_\lambda \Bigl(\sum_{i=1}^k U_i\Bigr)=\sum_{i=1}^kJ_\lambda(U_i)+\frac{1}{2}\sum_{i=1}^k \sum_{j\neq i}\int_{\Omega} w_i^5\, U_j-\frac{1}{6}\int_{\Omega}\Bigl[\Bigl(\sum_{i=1}^k U_i\Bigr)^6-\sum_{i=1}^k U_i^6\Bigr].
\label{Jlambda}
\end{equation}
Let $\rho\in(0,\delta/2)$ and denote by
\[
\mathcal{O}_\rho=\Omega\setminus\cup_{j=1}^kB_{\rho}(\zeta_j).
\]
Let us decompose
\begin{equation}
\label{Ei}
\int_{\Omega}\Bigl[\Bigl(\sum_{i=1}^k U_i\Bigr)^6-\sum_{i=1}^k U_i^6\Bigr]=\sum_{i=1}^k\int_{B_{\rho}(\zeta_i)}
E_i+\sum_{i=1}^k
\int_{\mathcal{O}_\rho}
E_i,
\end{equation}
where
\begin{align}
\notag
E_i:=&\Bigl[(U_i+Q_i)^6-U_i^6\Bigr]-\sum_{j\neq i} U_j^6\\
=&6\,(U_i^5\,Q_i+U_i\,Q_i^5)+15\,(U_i^4\,Q_i^2+Q_i^2\,U_i^4)
+20\,U_i^3\,Q_i^3-\sum_{j\neq i} U_j^6.
\label{U6}
\end{align}
and
$Q_i:=\sum_{j\neq i} U_j$.

From now on, we write simply $O(\mu^r)$ to indicate that some function is of the order of $(\mu_1+\ldots+\mu_k)^r$ for any $r>0$.

Notice that, if $s+t=6$,
\[
\mathcal{R}_{i,j}^{s,t}:=\int_{\mathcal{O_\rho}} U_i^s\,U_j^t=O(\mu^3).
\]
If, additionally, $s>t$,
\[
\tilde{\mathcal{R}}_{i,j}^{s,t}:=\int_{B_\rho(\zeta_i)} U_i^t\,U_j^s=O(\mu^3).
\]
This implies, in particular, that $\int_{\mathcal{O}_\rho}
E_i=O(\mu^3)$ and that $\int_{B_{\rho}(\zeta_i)}U_j^6=O(\mu^3)$.

(i)\,If $s=5$ and $t=1$, then we have
\begin{align}
\label{Ui5Uj}
\int_{B_\rho(\zeta_i)}U_i^5\,U_j&=
\int_{B_\rho(\zeta_i)}w_i^5\,U_j
+5\int_{B_\rho(\zeta_i)}w_i^4\,\pi_i\,U_j+
\mathcal{R}_{i,j}^1,
\end{align}
where
\[
\mathcal{R}_{ij}^1:=20
\int_0^1 d\tau\,(1-\tau)
\int_{B_\rho(\zeta_i)}(w_i+\tau\pi_i)^3\,\pi_i^2\,U_j.
\]
Using the change of variable $x=\zeta_i+\mu_i z$
and calling $B_{\mu_i}=B_{\frac{\rho}{\mu_i}}(0)$
we find that
\[
\int_{B_\rho(\zeta_i)}w_i^5\,U_j\,dx=
\mu_i^{\frac{1}{2}}\mu_j^{\frac{1}{2}}\int_{B_{\mu_i}}U^5(z)\,\mu_j^{-\frac{1}{2}}\,U_j(\zeta_i+\mu_i z)\,dz.
\]
By Lemma \ref{Uiexp} we have
\[
\mu_j^{-\frac{1}{2}}\,U_j(\zeta_i+\mu_i z)=4\pi\alpha_3\,G_{\lambda}(\zeta_i+\mu_i z,\zeta_j)+\mu_j^{2-\sigma}\hat{\theta}(\mu_j,\zeta_i+\mu_i z,\zeta_j).
\]
We expand
\begin{equation}
\label{Greenexp}
G_\lambda(\zeta_i+\mu_i z,\zeta_j)=
G_\lambda(\zeta_i,\zeta_j)+
\mu_i\, {\bf c}\cdot z+\theta_2(\zeta_i+\mu_i z,\zeta_j),
\end{equation}
where ${\bf c}=D_1G_\lambda(\zeta_i,\zeta_j)$ and $\vert\theta_2(\zeta_i+\mu_i z,\zeta_j)\vert\leq C\mu_i^2\,\vert z\vert^2.$

\noindent
By symmetry,
\[
\int_{B_{\mu_i}}({\bf c}\cdot z)\,U^5(z)\,dz=0
\]
an so,
\begin{align}
\int_{B_\rho(\zeta_i)}w_i^5\,U_j\,dy=&4\pi\alpha_3 \,
\mu_i^{\frac{1}{2}}\,\mu_j^{\frac{1}{2}}\,G_\lambda(\zeta_i,\zeta_j)\int_{\R^3}U^5(z)\,dz+\mathcal{R}_{i,j}\notag
\\
=&2a_1\,
\mu_i^{\frac{1}{2}}\,\mu_j^{\frac{1}{2}}\,G_\lambda(\zeta_i,\zeta_j)+\mathcal{R}_{i,j}^2,
\label{wi5Uj}
\end{align}
where
$
a_1:= 2\pi\,\alpha_3\int_{\R^3}U^5
$ and
\begin{align*}
\mathcal{R}_{i,j}^2:=&-4\pi\alpha_3 \,
\mu_i^{\frac{1}{2}}\,\mu_j^{\frac{1}{2}}\,G_\lambda(\zeta_i,\zeta_j)\int_{\R^3\setminus B_{\mu_i}}U^5(z)\,dz\\
&+
4\pi\alpha_3 \,
\mu_i^{\frac{1}{2}}\,\mu_j^{\frac{1}{2}}\,\int_{B_{\mu_i}}U^5(z)\,\theta_2(\zeta_i+\mu_i z,\zeta_j)\,dz\\
&+\mu_i^{\frac{1}{2}}\,\mu_j^{\frac{1}{2}}\int_{B_{\mu_i}}\mu_j^{2-\sigma}\,U^5(z)\,\hat{\theta}(\mu_j,\zeta_i+\mu_i z,\zeta_j)
\,dz.
\end{align*}
From Lemma~\ref{lemma22} and \cite[Appendix]{DPDM}, we have the following expansions, for any $\sigma > 0$, as $\mu_i\rightarrow 0$
\begin{equation*}
\mu_i^{-\frac{1}{2}}\pi_i(\zeta_i+\mu_iz)=-4\pi\alpha_3\,H_{\lambda}(\zeta_i+\mu_iz,\zeta_i)+\mu_i\,\mathcal{D}_0(z)+\mu_i^{2-\sigma}\,\theta(\mu_i,\zeta_i+\mu_iz,\zeta_i)
\end{equation*}
\begin{equation*}
H_{\lambda}(\zeta_i+\mu_iz,\zeta_i)=g_{\lambda}(\zeta_i)+\frac{\lambda}{8\pi}\,\mu_i\vert z\vert
+\theta_0(\zeta_i,\zeta_i+\mu_iz)
\end{equation*}
where $\theta_0$ is a function of class $C^2$ with $\theta_0(\zeta_i,\zeta_i)=0$.

\noindent
The above expressions, combined with
Lemma \ref{Uiexp} and \eqref{Greenexp}, gives
\begin{align}
\nonumber
\int_{B_\rho(\zeta_i)}w_i^4\,\pi_i\,U_j=&
\mu_i^{\frac{3}{2}}\,\mu_j^{\frac{1}{2}}\int_{B_{\mu_i}}U^4(z)\,\mu_i^{-\frac{1}{2}}\pi_i(\zeta_i+\mu_i z)\,\mu_j^{-\frac{1}{2}}U_j(\zeta_i+\mu_i z)\,dz
\\
\notag
&=-\mu_i^{\frac{3}{2}}\,\mu_j^{\frac{1}{2}}(4\pi\alpha_3)^2 \,g_\lambda(\zeta_i)\,
G_\lambda(\zeta_i,\zeta_j)\int_{\R^3}U^4(z)\,dz+\mathcal{R}_3\\
&=-\frac{2}{5}\,a_3\,\mu_i^{\frac{3}{2}}\,\mu_j^{\frac{1}{2}}\, g_\lambda(\zeta_i)\,
G_\lambda(\zeta_i,\zeta_j)+\mathcal{R}_{i,j}^3,
\label{wi4piiUj}
\end{align}
where $a_3:=\frac{5}{2}(4\pi\alpha_3)^2\int_{\R^3}U^4.$

From \eqref{Ui5Uj}, \eqref{wi5Uj} and \eqref{wi4piiUj}, we get
\begin{align}
\label{Ui5UjF}
\int_{B_\rho(\zeta_i)}U_i^5\,U_j=&
2\,a_1\,
\mu_i^{\frac{1}{2}}\mu_j^{\frac{1}{2}}G_\lambda(\zeta_i,\zeta_j)
-2\,a_3\,\mu_i^{\frac{3}{2}}\,\mu_j^{\frac{1}{2}} g_\lambda(\zeta_i)\,
G_\lambda(\zeta_i,\zeta_j)
\\&+
\mathcal{R}_{i,j}^1+\mathcal{R}_{i,j}^2+5\,\mathcal{R}_{i,j}^3,
\mathcal{R}_{i,j}^{5,1}.
\notag
\end{align}

(ii)\, If $s=4$ and $t=2$, we have
\begin{align*}
\int_{B_\rho(\zeta_i)}U_i^4\,U_j\,U_m&=
\int_{B_\rho(\zeta_i)}w_i^4\,U_j\,U_m+
\mathcal{R}_{i,j,m}^5,
\end{align*}
where
\[
\mathcal{R}_{i,j,m}^5:=4
\int_0^1 d\tau\,
\int_{B_\rho(\zeta_i)}(w_i+\tau\pi_i)^3\,\pi_i\,U_j\,U_m.
\]
From Lemma \ref{lemma22}, Lemma \ref{Uiexp}, and \eqref{Greenexp}, we get
\begin{align*}
\int_{B_\rho(\zeta_i)}w_i^4\,U_j\,U_m=&
\mu_i\,\mu_j\,\mu_m\int_{B_{\mu_i}}U^4(z)\,
\Bigl(\mu_j^{-\frac{1}{2}}\,U_j(\zeta_i+\mu_i z)\Bigr)\,
\Bigl(\mu_m^{-\frac{1}{2}}\,U_m(\zeta_i+\mu_i z)\Bigr)\,dz\\
&=\mu_i\,\mu_j\,\mu_m\,(4\pi\alpha_3)^2 \,
G_\lambda(\zeta_i,\zeta_j)\,G_\lambda(\zeta_i,\zeta_m)\int_{\R^3}U^4(z)\,dz+\mathcal{R}_{i,j,m}^6\\
&=\frac{2}{5}\,a_3\,\mu_i\,\mu_j\,\mu_m\,
G_\lambda(\zeta_i,\zeta_j)\,G_\lambda(\zeta_i,\zeta_m)+\mathcal{R}_{i,j,m}^6.
\end{align*}
Therefore,
\begin{align}
\int_{B_\rho(\zeta_i)}U_i^4\,U_j\,U_m&=
\frac{2}{5}\,a_3\,\mu_i\,\mu_j\,\mu_m\,
G_\lambda(\zeta_i,\zeta_j)\,G_\lambda(\zeta_i,\zeta_m)+\mathcal{R}_{i,j,m}^5
+\mathcal{R}_{i,j,m}^6.
\label{Ui4UjUm}
\end{align}

(iiii)\, If $s=3$ and $t=3$, we have
\begin{equation*}
\int_{B_\rho(\zeta_i)}U_i^3\,U_j^3=
\mathcal{R}_{i,j}^{8},
\end{equation*}
where
\begin{align*}
\mathcal{R}_{i,j}^8:=&\int_{B_\rho(\zeta_i)}w_i^3\,U_j^3+3\int_0^1 ds\,
\int_{B_\rho(\zeta_i)}(w_i+s\pi_i)^2\,\pi_i\,U_j^3.
\end{align*}
To analyse the size of the remainders $\mathcal{R}_{i,j}^\ell$ we proceed as in \cite{DPDM}. We have the following
\begin{equation*}
\frac{\partial^{m+n}}{\partial\zeta^m\partial\mu^n}\mathcal{R}_{i,j}^\ell=O(\mu^{3-(n+\sigma)})
\end{equation*}
for each $m=0,1$, $n=0,1,2$, $m+n\leq 2$,
$\ell=1,\ldots, 8$,
uniformly on all small $(\mu,\zeta)\in \Gamma_\delta$.

\noindent
Analogous statements hold true for ${\mathcal{R}}_{i,j}^{s,t}$ and $\tilde{\mathcal{R}}_{i,j}^{s,t}$ with $s+t=6$.

From \eqref{U6} and the previous analisys
we get that
\begin{align*}
\int_{B_{\rho}(\zeta_i)}
E_i
&=6\,\int_{B_{\rho}(\zeta_i)}U_i^5\,Q_i+15\int_{B_{\rho}(\zeta_i)}U_i^4\,Q_i^2+\mathcal{R}\\
&=
6\,\sum_{j\neq i}\int_{B_{\rho}(\zeta_i)}U_i^5\,U_j+15\sum_{j\neq i}\sum_{m\neq i}\int_{B_{\rho}(\zeta_i)}U_i^4\, U_j\, U_m+\mathcal{R}.
\end{align*}
This expression together with
\eqref{Ui5UjF} and \eqref{Ui4UjUm} yields
\begin{align*}
\int_{B_{\rho}(\zeta_i)}
E_i
&=
6\,\sum_{j\neq i}\Bigl[
2\,a_1\,
\mu_i^{\frac{1}{2}}\mu_j^{\frac{1}{2}}G_\lambda(\zeta_i,\zeta_j)
-2\,a_3\,\mu_i^{\frac{3}{2}}\,\mu_j^{\frac{1}{2}} g_\lambda(\zeta_i)\,
G_\lambda(\zeta_i,\zeta_j)\Bigr]\\
&+6\sum_{j\neq i}\sum_{m\neq i}
\Bigl[
a_3\,\mu_i\,\mu_j\,\mu_m\,
G_\lambda(\zeta_i,\zeta_j)\,G_\lambda(\zeta_i,\zeta_m)
\Bigr]\\
&=
6\,\sum_{j\neq i}\Bigl[
2\,a_1\,
\mu_i^{\frac{1}{2}}\mu_j^{\frac{1}{2}}G_\lambda(\zeta_i,\zeta_j)
-2\,a_3\,\mu_i^{\frac{3}{2}}\,\mu_j^{\frac{1}{2}} g_\lambda(\zeta_i)\,
G_\lambda(\zeta_i,\zeta_j)\Bigr]\\
&+6\,a_3\,\mu_i\Bigl(\sum_{j\neq i}
\mu_j\,
G_\lambda(\zeta_i,\zeta_j)
\Bigr)^2.
\end{align*}
Combining relations \eqref{Jlambda}, \eqref {Ei}, \eqref{U6},
\eqref{wi5Uj}, Lemma \ref{lemma21} and the above expression we get
the conclusion.
For the statement of this lemma $\theta_\lambda^{(1)}$ is defined as the sum of all remainders.

The formula
\[
\int_0^\infty\Bigl(\frac{r}{1+r^2} \Bigr)^q\frac{dr}{r^{\alpha+1}}=\frac{\Gamma\bigl(\frac{q-\alpha}{2}\bigr)\,\Gamma\bigl(\frac{q+\alpha}{2}\bigr)}{2\,\Gamma(q)}
\]
yields that
\[
a_1=8(\alpha_3\pi)^2,
\quad
a_3=120\,(\alpha_3\pi^2)^2.
\]
\end{proof}

\section{The linear problem}
\label{sectLinear}

Let $u$ be a solution of $(\wp_\lambda)$. For
$\varepsilon>0$, we define
\[
v(y) = \varepsilon^{1/2} u(\varepsilon y) .
\]
Then $v$ solves the boundary value problem
\begin{equation}
\label{maineqreesc}
\left\{
\begin{array}{rlll}
\Delta v+\varepsilon^2\,\lambda\, v&=&-v^5&\text{in } \Omega_\varepsilon,
\\
v&>&0&\text{in }\Omega_\varepsilon,
\\
v&=&0&\text{on }\partial\Omega_\varepsilon,
\end{array}
\right.
\end{equation}
where $\Omega_\varepsilon=\varepsilon^{-1}\,\Omega$.
Thus finding a solution of $(\wp_\lambda)$  which is a small perturbation of $\sum_{i=1}^k U_i $ is equivalent to finding a solution of $(\wp_\lambda)$ of the form
\[
\sum_{i=1}^k V_i  +\phi,
\]
where
\begin{align*}
V_i(y) &= \varepsilon^{ \frac{1}{2}} U_i( \varepsilon y)
= w_{\mu_i^{\prime},\zeta_i^{\prime}}(y)+\varepsilon^{\frac{1}{2}}\pi_i(\varepsilon\, y)
\quad y \in \Omega_\varepsilon,
\end{align*}
for $i=1,\ldots, k$, and $\phi$ is small in some appropriate sense.

Notice that $V_i$ satisfies
\begin{equation*}
\left\{
\begin{array}{rlll}
\Delta V_i+\varepsilon^2 \,\lambda \,V_i&=&-w_{\mu_i^{\prime},\zeta_i^{\prime}}^5&\text{in } \Omega_\varepsilon,
\\
V_i&=&0&\text{on }\partial\Omega_\varepsilon,
\end{array}
\right.
\end{equation*}
where
\begin{align}
\label{muiPrime}
\mu_i^{\prime}=\frac{\mu_i}{\varepsilon},\quad \zeta_i^{\prime}=\frac{\zeta_i}{\varepsilon}.
\end{align}
Then solving \eqref{maineqreesc} is equivalent to finding $\phi$ such that,
\begin{equation}
\label{linearizedproblem}
\left\{
\begin{array}{rlll}
L(\phi)&=&-N(\phi)-E&\text{in } \Omega_\varepsilon,
\\
\phi&=&0&\text{on }\partial\Omega_\varepsilon,
\end{array}
\right.
\end{equation}
where
\[
L(\phi)=\Delta\phi+\varepsilon^2\,\lambda\, \phi+5V^4\,\phi,
\]
\[
N(\phi)=(V+\phi)^5-V^5-5V^4\,\phi,
\]
\begin{equation}
\label{error}
E=V^5-
\sum_{i=1}^k   w_{\mu_i^{\prime},\zeta_i^{\prime}}^5.
\end{equation}
and
\begin{align}
\label{defV}
V = \sum_{i=1}^k V_i.
\end{align}
In what follows,
the canonical basis of $\R^3$ will be denoted by
\[
\textrm{e}_1=(1,0,0), \quad \textrm{e}_2=(0,1,0)\quad \textrm{e}_3=(0,0,1).
\]
Let $z_{i,j}$, $i=1,2,$ be given by
\begin{equation}
\left\{
\label{zij}
\begin{array}{rcl}
z_{i,j}(y)&=&D_{\zeta_i^{\prime}} w_{\mu_i^{\prime},\zeta_i^{\prime}}(y)\cdot \textrm{e}_j\quad j=1,2,3\\
z_{i,4}(y)&=&\frac{\partial\,w_{\mu_i^{\prime},\zeta_i^{\prime}}}{\partial\,\mu_i^{\prime}}(y).
\end{array}
\right.
\end{equation}
We recall that for each $i$, the functions $z_{i,j}$ for $j=1,...,4$, span the space
of all bounded solutions of the linearized problem:
\[
 \Delta z+5\,w_{\mu_i^{\prime},\zeta_i^{\prime}}^4\, z=0\quad\text{ in } \R^3.
 \]
A proof of this fact can be found for instance in \cite{Rey}.

Observe that
\[
\int_{\R^3} w_{\mu_i^{\prime},\zeta_i^{\prime}}^4 z_{i,j}\,z_{i,l}=0\quad\text{if }j\neq l.
\]
In order to study the operator $L$, the key idea is that, as $\varepsilon\rightarrow 0$, the linear operator $L$ is close to being the sum of
\[
\Delta +5w_{\mu_i^{\prime},\zeta_i^{\prime}}^4,
 \]
$i=1,\ldots,k$.


Rather than solving \eqref{linearizedproblem} directly, we will look for a solution of the following problem first: Find a function $\phi$ such that for certain constants $c_{i,j}$, $i=1,2$, $j=1,2,3,4$,
\begin{equation}
\label{linearizedproblemproj}
\left\{
\begin{array}{rlll}
L(\phi)&=&-N(\phi)-E+\sum_{i,j}c_{ij}\,w_{\mu_i^{\prime},\zeta_i^{\prime}}^4\,z_{ij}
&\text{in } \Omega_\varepsilon,
\\
\phi&=&0&\text{on }\partial\Omega_\varepsilon,
\\
\int_{\Omega_\varepsilon}w_{\mu_i^{\prime},\zeta_i^{\prime}}^4\,z_{ij}\,\phi&=&0&\text{for all  }i,j.
\end{array}
\right.
\end{equation}
After this is done, the remaining task is to adjust the parameters $\zeta_i^\prime, \mu_i^\prime$ in such a way that all constants $c_{ij}=0$.

In order to solve problem \eqref{linearizedproblemproj} it is necessary to understand its linear part. Given a function $h$ we consider the problem of finding $\phi$ and real numbers $c_{ij}$ such that
\begin{equation}
\label{linpart}
\left\{
\begin{array}{rlll}
L(\phi)&=&h+\sum_{i,j}c_{ij}\,w_{\mu_i^{\prime},\zeta_i^{\prime}}^4\,z_{ij}
&\text{in } \Omega_\varepsilon,
\\
\phi&=&0&\text{on }\partial\Omega_\varepsilon,
\\
\int_{\Omega_\varepsilon}w_{\mu_i^{\prime},\zeta_i^{\prime}}^4\,z_{ij}\,\phi&=&0&\text{for all  }i,j.
\end{array}
\right.
\end{equation}
We would like to show that this problem is uniquely solvable with uniform bounds in suitable functional spaces.
To this  end, it is convenient to introduce the following weighted norms.

\noindent
Given a fixed number $\nu\in(0,1)$,
we define
\begin{align*}
\Vert f\Vert_{\ast}
&=\sup_{y\in\Omega_\epsilon}
\Bigl(
\omega(y)^{-\nu}\,\vert f(y)\vert
+\omega(y)^{-\nu-1}\,\vert \nabla f(y)\vert
\Bigr)
\\
\Vert f\Vert_{\ast\ast}
&=\sup_{y\in\Omega_\epsilon}
\omega(y)^{-(2+\nu)}\,\vert f(y) \vert,
\end{align*}
where
\[
\omega(y)=\sum_{i=1}^k\bigl(1+\vert y-\zeta_i^{\prime}\vert\bigr)^{-1}.
\]

\begin{prop}
\label{solvability}
Let $0<\alpha<1$. Let $\delta>0$ be given.
Then there exist a positive number $\varepsilon_0$ and a constant $C>0$ such that if \,$0<\varepsilon<\varepsilon_0$, and
\begin{align}
\label{parametros}
\vert \zeta_i^{\prime}-\zeta_j^{\prime} \vert >\frac{\delta}{\varepsilon}, \ i\not=j; \quad
dist(\zeta_i^{\prime},\partial\Omega_{\varepsilon})>\frac{\delta}{\varepsilon}
\text{ and }
\delta<\mu_i^{\prime}<\delta^{-1},\  i=1,\ldots,k,
\end{align}
then for any $h\in C^{0,\alpha}(\Omega_\varepsilon)$ with $\Vert h\Vert_{\ast\ast}<\infty$, problem \eqref{linpart} admits a unique solution $\phi=T(h)\in C^{2,\alpha}(\Omega_\varepsilon)$. Besides,
\begin{equation}
\label{cotas}
\Vert T(h)\Vert_{\ast}\leq C\,\Vert h\Vert_{\ast\ast} \quad\text{and}\quad
\vert c_{ij}\vert\leq C\,\Vert h\Vert_{\ast\ast},\,\,i=1,\ldots,k,\,\, j=1,2,3,4.
\end{equation}
\end{prop}
Here and in the rest of this paper, we denote by $C$ a positive constant that may change from line to line but is always independent of $\varepsilon$.

For the proof of the previous proposition
we need the following a priori estimate:
\begin{lemma}
\label{lemmaCotaApriori}
Let $\delta>0$ be a given small number.
Assume the existence of sequences $(\varepsilon_n)_{n\in\mathbb{N}}$, $(\zeta^{\prime}_{i,n})_{n\in\mathbb{N}}$,, $(\mu^{\prime}_{i,n})_{n\in\mathbb{N}}$ such that $\varepsilon_n> 0$,
$\varepsilon_n\rightarrow 0$,
\[\vert \zeta^{\prime}_{i,n}-\zeta^{\prime}_{j,n} \vert >\frac{\delta}{\varepsilon_n}, \ i\not=j;
\quad
dist(\zeta^{\prime}_{i,n},\partial\Omega_{\varepsilon_n})>\frac{\delta}{\varepsilon_n}
\text{ and }
\delta<\mu^{\prime}_{i,n}<\delta^{-1},\ i=1,\ldots,k,
\]
and for certain functions $\phi_n$ and $h_n$ with
$\Vert h_n\Vert_{\ast\ast}\rightarrow 0$ and scalars
$c_{ij}^n$, $i=1,\ldots,k$, $j=1,2,3,4$, one has
\begin{equation}
\label{linpartn}
\left\{
\begin{array}{rlll}
L(\phi_n)&=&h_n+\sum_{i,j}c_{ij}^n\,w_{\mu_{i,n}^{\prime},\zeta_{i,n}^{\prime}}^4\,z_{ij}^n
&\text{in } \Omega_{\varepsilon_n},
\\
\phi_n&=&0&\text{on }\partial\Omega_{\varepsilon_n},
\\
\int_{\Omega_{\varepsilon_n}}w_{\mu_{i,n}^{\prime},\zeta_{i,n}^{\prime}}^4\,z_{ij}^n\,\phi_n&=&0&\text{for all  }i,j,
\end{array}
\right.
\end{equation}
where the functions $z_{ij}^n$
are defined as in \eqref{zij} for $\zeta_{i,n}^{\prime}$ and $\mu_{i,n}^{\prime}$.
Then
\[
\lim_{n\rightarrow\infty}\Vert \phi_n\Vert_{\ast}=0.
\]
\end{lemma}
\begin{proof}
Arguing by contradiction, we may assume that $\Vert \phi_n\Vert_{\ast}=1$.
We shall establish first the weaker assertion that
\[
\lim_{n\rightarrow\infty}\Vert \phi_n\Vert_{\infty}=0.
\]
Let us assume, for contradiction, that except possibly for a subsequence
\begin{equation}
\label{contra}
\lim_{n\rightarrow\infty}\Vert \phi_n\Vert_{\infty}=\gamma,\quad\text{ with }0<\gamma\leq 1.
\end{equation}
We consider a cut-off function $\eta\in C^{\infty}(\R)$ with
\[
\eta(s)\equiv 1 \quad\text{for } s\leq \frac{\delta}{2},\quad
\eta(s)\equiv 0 \quad\text{for } s\geq \delta.
\]
We define
\begin{equation}
\label{zklcortada}
{\bf z}_{kl}^n(y):=\eta(2\,\varepsilon_n\,\vert y-\zeta_{k,n}^{\prime}\vert)\,z_{kl}^n(y).
\end{equation}
Testing \eqref{linpartn} against ${\bf z}_{kl}^n$ and integrating by parts twice we get the following relation
\[
\sum_{i,j}c_{ij}^n\,
\int_{\Omega_{\varepsilon_n}}
w_{\mu_{i,n}^{\prime},\zeta_{i,n}^{\prime}}^4\,z_{ij}^n\,{\bf z}_{kl}^n
=
\int_{\Omega_{\varepsilon_n}}
L({\bf z}_{kl}^n)\,\phi_n-\int_{\Omega_{\varepsilon_n}}h_n\,{\bf z}_{kl}^n.
\]
Since $z_{kl}^n$ lies on the kernel of
\[
L_{k}:=\Delta +5w_{\mu_k^{\prime},\zeta_k^{\prime}}^4,
\]
writing $L({\bf z}_{kl}^n)=L({\bf z}_{kl}^n)-L_k(z_{kl}^n)$,
it is easy to check that
\[
\Bigl\vert
\int_{\Omega_{\varepsilon_n}}
L({\bf z}_{kl}^n)\,\phi_n
\Bigr\vert =o(1)\, \Vert \phi_n\Vert_{\ast}\quad\text{for } l=1,2,3,4.
\]
To obtain the last estimate, we take into account the effect of the Laplace operator on the cut-off function $\eta$ which is used to define ${\bf z}_{kl}^n$ and the effect of the difference between the two potentials $V^4$ and $w_{\mu_k^{\prime},\zeta_k^{\prime}}^4$ which appear respectively in the definition of $L$ and $L_{k}$.

On the other hand, a straightforward computation yields
\[
\Bigl\vert
\int_{\Omega_{\varepsilon_n}}h_n\,{\bf z}_{kl}^n
\Bigr\vert \leq C\, \Vert h_n\Vert_{\ast\ast}.
\]
Finally, since
\[
\int_{\Omega_{\varepsilon_n}}
w_{\mu_{i,n}^{\prime},\zeta_{i,n}^{\prime}}^4\,z_{ij}^n\,{\bf z}_{kl}^n=C\,\delta_{i,k}\,\delta_{j,l}+o(1) \quad \text{with } \delta_{i,k}=
\left\{
\begin{array}{ll}
1 &\text{ if }i=k\\
0 &\text{ if }i\neq k,
\end{array}
\right.
\]
we conclude that
\[
\lim_{n\rightarrow\infty}c_{ij}^n=0, \quad \text{for all }i,j.
\]
Now, let $y_n \in \Omega_{\varepsilon_n}$ be such that $\phi_n(y_n)=\gamma$, so that $\phi_n$ attains its absolute maximum value at this point.
Since $\Vert \phi_n\Vert_{\ast}=1,$ there is a radius $R>0$ and $i\in\{1,\ldots,k\}$ such that, for $n$ large enough,
\[\vert y_n-\zeta_{i,n}^\prime\vert\leq R
.\]
Defining $\tilde{\phi}_n(y)=\phi_n(y+\zeta_{i,n}^\prime)$ and using elliptic estimates together with Ascoli-Arzela's theorem, we have that, up to a subsequence,  $\tilde{\phi}_n$ converges uniformly  over compacts to a nontrivial bounded solution $\tilde{\phi}$ of
\begin{equation*}
\left\{
\begin{array}{rlll}
-\Delta \,\tilde{\phi}+5\,w_{\mu_{i}^{\prime},0}^4\,\tilde{\phi}&=&0&\text{in }\R^3,
\\
\int_{\R^3}w_{\mu_{i}^{\prime},0}^4\,z_{0,j}\tilde{\phi}&=&0 &\text{for }j=1,2,3,4,
\end{array}
\right.
\end{equation*}
which is bounded by a constant times $\vert y\vert^{-1}$.
Here
$z_{0,j}$ is defined as
in \eqref{zij} taking $\zeta_{i}^{\prime}=0$ and $\mu_{i}^{\prime}:=\lim_{n\rightarrow\infty}\mu_{i,n}^{\prime}$ (up to subsequence).
From the assumptions, it follows that $\delta\leq \mu_{i}^{\prime}\leq \delta^{-1}$.

Now, taking into account that
the solution $w_{\mu_{i}^{\prime},0}$ is nondegenerate, the above implies that
$\tilde{\phi}=\sum_{j=1}^4 \alpha_j\,z_{0,j}(y)$
and then, from the orthogonality conditions
we can deduce that $\alpha_j=0$ for $j=1,2,3,4.$
From here we obtain $\tilde{\phi}\equiv 0$, which contradicts \eqref{contra}.
This proves that $\lim_{n\rightarrow\infty}\Vert \phi_n\Vert_{\infty}=0$.

Next we shall establish that $\Vert \phi_n\Vert_{\normadora}\rightarrow 0$ where
\begin{align*}
\Vert \phi \Vert_{\normadora}
&=\sup_{y\in\Omega_\epsilon}
\omega(y)^{-\nu}\,\vert \phi(y)\vert .
\end{align*}
Defining
\[
\psi_n(x)=\frac{1}{\varepsilon_n^\nu}\,\phi_n\Bigl(\frac{x}{\varepsilon_n}\Bigr),\quad x\in\Omega
\]
we have that $\psi_n$ satisfies
\begin{equation*}
\left\{
\begin{array}{rllll}
\Delta \,\psi_n+\lambda \,\psi_n&=&\varepsilon_n^{-(2+\nu)}\Bigl\{
&-5 \varepsilon_n^{1/2}
\left( \varepsilon_n^{1/2}
\sum_{i=1}^k
U_{\mu_{i,n},\zeta_{i,n}}\right)^4\,\varepsilon_n^{\nu}\,\psi_n\\
&&&+
g_n
+\sum_{i,j}c_{ij}^n\,\varepsilon_n^{2}\,w_{\mu_{i,n},\zeta_{i,n}}^4\,Z_{ij}^n\Bigr
\}
&\text{in } \Omega,
\\
\psi_n&=&0&&\text{on }\partial\Omega,
\end{array}
\right.
\end{equation*}
where $\mu_{i,n}=\varepsilon_n\,\mu_{i,n}^{\prime}$,
$\zeta_{i,n}=\varepsilon_n\,\zeta_{i,n}^{\prime}$,
$g_n(x)=h_n\bigl(\frac{x}{\varepsilon_n}\bigr)$ and $Z_{ij}^n(x)=z_{ij}^n\bigl(\frac{x}{\varepsilon_n}\bigr)$.

Let $\zeta_i\in\Omega$ be such that, after passing to a subsequence, $\vert\zeta_{i,n}-\zeta_i\vert\leq\frac{\delta}{4}$ for all $n\in\mathbb{N}$. Notice that, by the assumptions, $B_{\frac{\delta}{4}}(\zeta_i)\subset\Omega$ and $B_{\frac{\delta}{4}}(\zeta_i)\cap B_{\frac{\delta}{4}}(\zeta_j)=\emptyset$ for $i\not=j$.
From the assumption $\Vert \phi_n\Vert_{*} =  1$ we deduce that
\begin{equation*}
\vert \psi_n(x)\vert\leq \biggl(
\sum_{i=1}^k
\frac{1}{\varepsilon_n+\vert x-\zeta_{i,n}\vert} \biggr)^{\nu} ,
\quad \forall x\in \Omega.
\end{equation*}
Since $\lim_{n\rightarrow\infty}\Vert h_n\Vert_{\ast\ast}\rightarrow 0$,
\[
\vert g_n(x)\vert\leq o(1)\,\varepsilon_n^{2+\nu}\,\biggl(
\sum_{i=1}^k
\frac{1}{\varepsilon_n+\vert x-\zeta_{i,n}\vert}
\biggr)^{2+\nu}
\quad \text{for }x\in \Omega.
\]
From Lemma \ref{Uiexp} we know that,
away from
$\zeta_{i,n}$,
\[
U_{\mu_{i,n},\zeta_{i,n}}(x)=C\,\varepsilon_{n}^{1/2}\,(1+o(1))\,G_{\lambda}(x,\zeta_{i,n}).
\]
Moreover, it is easy to see that also away from
$\zeta_{i,n}$,
\[
\varepsilon_n^{-\nu}\,
\sum_{j=1}^4 c_{ij}^n\,w_{\mu_{i,n},\zeta_{i,n}}^4\,Z_{ij}^n=o(1)\quad \text{as } \varepsilon_n\rightarrow 0,
\]
and so, a diagonal convergence argument allows us to conclude that $\psi_n(x)$ converges uniformly over compacts of $\bar{\Omega}\setminus\{\zeta_1,\ldots,\zeta_k\}$ to $\psi(x)$, a solution of
\[
-\Delta \,\psi+\lambda \,\psi=0
\quad\text{in } \Omega\setminus\{\zeta_1,\ldots,\zeta_k\},
\quad
\psi=0\quad\text{on }\partial\Omega ,
\]
which satisfies
\begin{equation*}
\vert \psi(x)\vert\leq \biggl(
\sum_{i=1}^k
\frac{1}{\vert x-\zeta_{i,n}\vert} \biggr)^{\nu} ,
\quad \forall x\in \Omega.
\end{equation*}
Thus $\psi$ has a removable singularity at all $\zeta_i$, $i=1,\ldots,k$, and we conclude that $\psi(x)=0$. Hence, over compacts of $\bar{\Omega}\setminus\{\zeta_1,\ldots,\zeta_k\}$,  $\vert \psi_n(x)\vert=o(1).$
In particular, this implies that, for all
$x\in\Omega\setminus
\bigl( \cup_{i=1}^k
B_{\frac{\delta}{4}}(\zeta_{i,n})  \bigr)$,
$
\vert \psi_n(x)\vert\leq o(1).
$
Thus we have
\begin{equation}
\label{57}
\vert \phi_n(y)\vert\leq o(1)\,\varepsilon_n^\nu,
\quad \text{for all } y\in \Omega_{\varepsilon_n}\setminus\Bigl(
\bigcup_{i=1}^k
B_{\frac{\delta}{4\varepsilon_n}}(\zeta_{i,n}^\prime)
\Bigr).
\end{equation}
Now, consider a fixed number $M$, such that $M<\frac{\delta}{4\,\varepsilon_n}$, for all $n\in\mathbb{N}$.

\noindent
Since $\Vert \phi_n\Vert_\infty=o(1)$,
\begin{equation}
\label{58}
\bigl(1+|y-\zeta_{i,n}^\prime|\bigr)^{\nu}
\vert \phi_n(y)\vert\leq o(1) \quad\text{for all }
y\in \overline{B_{M}(\zeta_{i,n}^\prime)}.
\end{equation}
We claim that
\begin{equation}
\label{59}
\bigl(1+|y-\zeta_{i,n}^\prime|\bigr)^{\nu}
\vert \phi_n(y)\vert\leq o(1) \quad\text{for all }
y\in A_{\varepsilon_n,M},
\end{equation}
where
 $A_{\varepsilon_n,M}:=B_{\frac{\delta}{4\,\varepsilon_n}}(\zeta_{i,n}^\prime)\setminus\overline{B_{M}(\zeta_{i,n}^\prime)}$.

The proof of this assertion relies on the fact that the operator $L$ satisfies the weak maximum principle in $A_{\varepsilon_n,M}$ in the following sense: if $u$ is bounded, continuous in $\overline{A_{\varepsilon_n,M}}$, $u\in H^1(A_{\varepsilon_n,M})$ and satisfies $L(u)\geq 0$ in $A_{\varepsilon_n,M}$ and $u\leq 0$ in $\partial\,A_{\varepsilon_n,M},$ then, choosing a larger $M$ if necessary, $u\leq 0$ in $A_{\varepsilon_n,M}$. We remark that this result is just a consequence of the fact that $L(\vert y-\zeta_{i,n}^\prime\vert^{-\nu})\leq 0$ in $A_{\varepsilon_n,M}$ provided that $M$ is large enough but independent of $n$.

\noindent
Next, we shall define an appropriate
barrier function. First we
observe that there exists $\eta_n^1\rightarrow 0$, as $\varepsilon_n\rightarrow 0$, such that
\begin{equation}
\label{Lphi}
\vert y-\zeta_{i,n}^{\prime}\vert^{2+\nu}\,\vert L(\phi_n)\vert\leq \eta_n^1 \quad
\text{in } A_{\varepsilon_n,M}.
\end{equation}
On the other hand, from
\eqref{57} we deduce the
existence of $\eta_n^2\rightarrow 0$, as $\varepsilon_n\rightarrow 0$, such that
\begin{equation}
\label{extrad}
\varepsilon_n^{-\nu}\vert \phi_n(y)\vert\leq \eta_n^2 \quad \text{ if } \vert y-\zeta_{i,n}^\prime\vert =\delta/4\varepsilon_n,
\end{equation}
and from \eqref{58} we deduce the existence of $\eta_n^3\rightarrow 0$, as $\varepsilon_n\rightarrow 0$, such that
\begin{equation}
\label{intrad}
M^{\nu}\vert \phi_n(y)\vert\leq \eta_n^3, \quad \text{if } \vert y-\zeta_{i,n}^\prime\vert =M.
\end{equation}
Setting $\eta_n=\max\{ \eta_n^1, \eta_n^2, \eta_n^3\}$ we find that
the function
\[
\varphi_n(y)=\eta_n\,\vert y-\zeta_{i,n}^{\prime}\vert^{-\nu}
\]
can be used for the intended comparison argument.

Indeed, for each $i=1,\ldots,k$ we can write
\begin{align*}
L(\vert y-\zeta_{i,n}^{\prime}\vert^{-\nu})
&=-\Bigl(
\nu\,(1-\nu)-\bigl(\varepsilon_n^2\,\lambda
+5(V_1+V_2)^4\bigr)\,\vert y-\zeta_{i,n}^{\prime}\vert^2
\Bigr)
\,\vert y-\zeta_{i,n}^{\prime}\vert^{-(2+\nu)}\\
&\leq -\frac{\nu\,(1-\nu)}{2} \,\vert y-\zeta_{i,n}^{\prime}\vert^{-(2+\nu)}
\end{align*}
provided $\vert y-\zeta_{i,n}^{\prime}\vert$ is large enough,
and then
\begin{equation*}
L(\varphi_n)\leq -\frac{\nu\,(1-\nu)}{2} \,\,\eta_n\,\vert y-\zeta_{i,n}^{\prime}\vert^{-(2+\nu)} \quad
\text{in } A_{\varepsilon_n,M}
\end{equation*}
provided $M$ is fixed large enough (independently of $n$).
This together with \eqref{Lphi} yields that $
\vert L(\phi_n)\vert\leq - C L(\varphi_n)
$ in $A_{\varepsilon_n,M}$. Moreover, it follows from
\eqref{extrad} and \eqref{intrad} that
$
\vert \phi_n(y)\vert\leq C \varphi_n(y)
$ on $\partial \,A_{\varepsilon_n,M}$
and thus the maximum principle allows us to conclude that \eqref{59} holds.

Thus, we have shown that $\|\phi_n \|_{\normadora} \to 0$ as $n\to\infty$.
A standard argument using an appropriate scaling and  elliptic estimates  shows that $\|\phi\|_*\to 0$ as $n\to \infty$, which contradicts the assumption  $\Vert \phi_n \Vert_*=1$.
\end{proof}

\begin{proof}[Proof of Proposition \ref{solvability}]
Let us consider the space:
\[
H=\Bigl\{
\phi\in H_0^1(\Omega_\varepsilon): \int_{\Omega_\varepsilon}w_{\mu_i^{\prime},\zeta_i^{\prime}}^4\,z_{ij}\,\phi=0, \, i=1,\ldots,k,\, j=1,2,3,4
\Bigr\}
\]
endowed with the inner product:
\[
[\phi,\psi]=\int_{\Omega_\varepsilon}
\nabla \phi \cdot\nabla \psi
-\varepsilon^2\,\lambda \int_{\Omega_\varepsilon}
\phi \,\psi.
\]
Problem \eqref{linpart} expressed in weak form is equivalent to that of finding a $\phi\in H$ such that
\[
[\phi,\psi]=\int_{\Omega_\varepsilon}\Bigl[5(V_1+V_2)^4 \phi -h-\Bigr]\,\psi\quad\text{for all }\psi\in H.
\]
With the aid of Riesz’s representation theorem, this equation gets rewritten in $H$ in the operational form $\phi = K(\phi) + \tilde{h}$, for certain $\tilde{h}\in H$, where $K$ is a compact operator in $H$.
Fredholm's alternative guarantees unique solvability of this problem for any $\tilde{h}$ provided that the homogeneous equation $\phi=K(\phi)$ has only the zero solution in $H$. Let us observe that this last equation is precisely equivalent to \eqref{linpart} with $h=0$. Thus existence of a unique solution follows.
Estimate \eqref{cotas} can be deduced from Lemma~\ref{lemmaCotaApriori}.
\end{proof}

It is important, for later purposes, to understand the differentiability of the operator
$T:h\mapsto \phi$ with respect to the variables $\mu_i^{\prime}$ and $\zeta_i^{\prime}$, $i=1,\ldots,k$,
for $\varepsilon$ fixed. That is, only the parameters $\mu_i$ and $\zeta_i$ are allowed to vary.

\begin{prop}
\label{propDerLinearOp}
Let $\mu^{\prime}:=(\mu_1^{\prime},\ldots,\mu_k^{\prime})$ and $\zeta^{\prime}:=(\zeta_1^{\prime},\ldots,\zeta_k^{\prime})$.
Under the conditions of Proposition \ref{solvability}, the map $T$ is of class $C^1$ and the derivative $D_{\mu^{\prime},\,\zeta^{\prime}}\,D_{\mu^{\prime}}T$  exists and is a continuous function.
Besides, we have
\[
\Vert D_{\mu^{\prime},\,\zeta^{\prime}}T(h)\Vert_{\ast}
+\Vert D_{\mu^{\prime},\,\zeta^{\prime}}\,D_{\mu^{\prime}}T(h)\Vert_{\ast}\leq C
\Vert h\Vert_{\ast\ast}.
\]
\end{prop}
\begin{proof}
Let us begin with differentiation with respect to $\zeta^\prime$. Since $\phi$ solves problem \eqref{linpart},
formal differentiation yields that $X_{n}:=\partial_{(\zeta^\prime)_n}\phi$, $n=1,\ldots,3k$, should satisfy
\begin{equation*}
L(X_{n})=
-5\,\bigl[ \partial_{(\zeta^\prime)_n}    V^4\bigl]\,\phi
+\sum_{i,j}c_{ij}^n\,w_{\mu_i^{\prime},\zeta_i^{\prime}}^4\,z_{ij}
+\sum_{i,j}c_{ij}\,\partial_{(\zeta^\prime)_n}\bigl[w_{\mu_i^{\prime},\zeta_i^{\prime}}^4\,z_{ij}\bigr]
\quad\text{in } \Omega_\varepsilon
\end{equation*}
together with
\begin{equation}
\label{ort}
\int_{\Omega_\varepsilon}X_{n}\,w_{\mu_i^{\prime},\zeta_i^{\prime}}^4\,z_{ij}+\int_{\Omega_\varepsilon}\phi\,\partial_{(\zeta^\prime)_n}\bigl[w_{\mu_i^{\prime},\zeta_i^{\prime}}^4\,z_{ij}\bigr]=0\quad\text{for   }j=1,2,3,4 ,
\end{equation}
where $c_{ij}^n=\partial_{(\zeta^\prime)_n}c_{ij}$.

\noindent
Let us consider constants $b_{ml}$
such that
\[
\int_{\Omega_\varepsilon}
\Bigl(
X_{n}-\sum_{m,l}b_{ml} \,{\bf z}_{ml}
\Bigr)\,w_{\mu_i^{\prime},\zeta_i^{\prime}}^4\,z_{ij}=0,
\]
where ${\bf z}_{ml}$ is defined in \eqref{zklcortada}.
From \eqref{ort} we get
\begin{equation*}
\sum_{m,l}b_{ml}
\int_{\Omega_\varepsilon}w_{\mu_i^{\prime},\zeta_i^{\prime}}^4\,z_{ij}\, {\bf z}_{ml}=-\int_{\Omega_\varepsilon}\partial_{(\zeta^\prime)_n}\bigl[w_{\mu_i^{\prime},\zeta_i^{\prime}}^4\,z_{ij}\bigr]\,\phi
\end{equation*}
for  $i=1,\ldots,k$, $j=1,2,3,4$.
Since this system is diagonal dominant with uniformly bounded coefficients, we see that it is uniquely solvable and that
\[
b_{ml}=O(\Vert \phi\Vert_\ast)
\]
uniformly on $\zeta^\prime$, $\mu^\prime$ in $\Omega_\varepsilon$.
On the other hand, it is not hard to check that
\[
\bigl\Vert \phi\,\partial_{(\zeta^\prime)_n} V^4 \bigr\Vert_{\ast\ast}\leq C\,
\Vert \phi\Vert_\ast.
\]
Recall now that from Proposition \ref{solvability} $c_{i,j}=O(\Vert h\Vert_{\ast\ast})$. Since besides
\[\Bigl\vert\partial_{(\zeta^\prime)_n}\bigl[w_{\mu_i^{\prime},\zeta_i^{\prime}}^4\, z_{ij}(x)\bigr]\Bigr\vert\leq C\, \bigl\vert y-\zeta^\prime_i\bigr\vert^{-7},\]
we get
\[\Bigl\Vert
\sum_{i,j}c_{ij}\,\partial_{(\zeta^\prime)_n}\bigl[
w_{\mu_i^{\prime},\zeta_i^{\prime}}^4\, z_{ij}
\bigr]\Bigr\Vert_{\ast\ast}\leq C\,
\Vert h\Vert_{\ast\ast}.
\]
Setting $X=X_{n}-\sum_{m,l}b_{ml} \,{\bf z}_{ml}$, we have that $X$ satisfies
\[
L(X)=f+\sum_{i,j}c_{ij}^n\,w_{\mu_i^{\prime},\zeta_i^{\prime}}^4\,z_{ij}\quad\text{in }\Omega_\varepsilon,
\]
where
\[
f=\sum_{m,l}b_{ml} \,L({\bf z}_{ml})-5\,\phi\,\partial_{(\zeta^\prime)_n}V^4
+\sum_{i,j}c_{ij}\,\partial_{(\zeta^\prime)_n}\bigl[w_{\mu_i^{\prime},\zeta_i^{\prime}}^4\,z_{ij}\bigr].
\]
The above estimates, together with the fact that $\Vert \phi\Vert_\ast\leq C\,
\Vert h\Vert_{\ast\ast}$ implies that
\[\Vert f \Vert_{\ast\ast}\leq C\,
\Vert h\Vert_{\ast\ast}.\]
Moreover, since $X\in H_0^1(\Omega)$
 and
\[
\int_{\Omega_\varepsilon}
X\,w_{\mu_i^{\prime},\zeta_i^{\prime}}^4\,z_{ij}=0
\quad\text{for all }i,j,
\]
we have that $X=T(f)$.
This computation is not just formal. Indeed,
arguing directly by definition, one gets that
\[
\partial_{(\zeta^\prime)_n}\phi=\sum_{m,l}b_{ml} \,{\bf z}_{ml}+T(f)\quad\text{and}\quad
\Vert \partial_{(\zeta^\prime)_n}\phi\Vert_{\ast}\leq C \,
\Vert h\Vert_{\ast\ast}.
\]
The corresponding result for differentiation with respect to the $\mu_i$'s  follows similarly. This concludes the proof.
\end{proof}

\section{The nonlinear problem}
\label{sectNonlinear}

In this section we consider the nonlinear problem
\eqref{linearizedproblemproj}, namely,
\begin{equation}
\label{linearizedproblemproj2}
\left\{
\begin{array}{rlll}
L(\phi)&=&-N(\phi)-E+\sum_{i,j}c_{ij}\,w_{\mu_i^{\prime},\zeta_i^{\prime}}^4\,z_{ij}
&\text{in } \Omega_\varepsilon,
\\
\phi&=&0&\text{on }\partial\Omega_\varepsilon,
\\
\int_{\Omega_\varepsilon}w_{\mu_i^{\prime},\zeta_i^{\prime}}^4\,z_{ij}\,\phi&=&0&\text{for all  }i,j ,
\end{array}
\right.
\end{equation}
and show that it has a small solution $\phi$ for $\varepsilon>0$ small enough.

We first obtain an estimate of the error $E$ defined in \eqref{error}.
Assuming \eqref{parametros} it is possible to show that $E$ satisfies
$
\|E\|_{**} \leq C \varepsilon.
$
However, for the proof of the main theorem, we require a stronger estimate.  In order to find it, we need  to impose certain extra assumptions on the parameters.

Let us use the notation
\[
\mu^{\frac{1}{2}} =
\left[
\begin{matrix}
\mu_1^{\frac{1}{2}}
\\
\vdots
\\
\mu_k^{\frac{1}{2}}
\end{matrix}
\right] \in \R^k.
\]
%

\begin{lemma}
\label{lemma-error}
Assuming that the parameters $\mu_i,\zeta_i$ satisfy \eqref{parametros}, where $\delta>0$ is fixed small, we have the existence of
$\varepsilon_1>0$, $C>0$, such that for all $\varepsilon \in (0,\varepsilon_1)$
\[
\| E \|_{**} \leq C ( \varepsilon^{\frac{1}{2}} |M_\lambda(\zeta) \mu^{\frac{1}{2}}| + \varepsilon^2) .
\]
\end{lemma}
\begin{proof}
We recall that
\begin{align*}
E(y)=\Bigl(
\sum_{i=1}^k
\bigl[
w_{\mu_i^{\prime},\zeta_i^{\prime}}(y)+\varepsilon^{\frac{1}{2}}\pi_i(\varepsilon\, y)
\bigr]
\Bigr)^5
- \sum_{i=1}^k
w_{\mu_i^{\prime},\zeta_i^{\prime}}^5(y)
,\quad y\in \Omega_{\varepsilon} .
\end{align*}
First we note that
\[
|E(y)|\leq C \varepsilon^5,\quad \text{if }
y \in \widetilde \Omega_\varepsilon := \Omega_\varepsilon \setminus \bigcup_{j=1}^k B_{\delta/\varepsilon}(\zeta_j') ,
\]
and this implies that
\begin{align}
\label{EcomplementB}
\sup_{y \in \widetilde \Omega_\varepsilon}  \omega(y)^{-(2+\nu)}   |E(y)| \leq C \varepsilon^{5-\nu}.
\end{align}
For $y \in B_{\delta/\varepsilon}(\zeta_i'))$ and $j\not=i$, thanks to Lemma~\ref{lemma22} we have
\[
\varepsilon^{\frac{1}{2}}\pi_i(\varepsilon\, y)  = O(\varepsilon), \quad
w_{\mu_j^{\prime},\zeta_j^{\prime}}(y)+\varepsilon^{\frac{1}{2}}\pi_j(\varepsilon\, y)
= O(\varepsilon).
\]
Hence, using Taylor's theorem  and the fact that $\mu_i = O ( \varepsilon )$ (which follows from \eqref{parametros}), we find that
\begin{align}
\nonumber
E(y)
& = 5  w_{\mu_i^{\prime},\zeta_i^{\prime}}(y)^4
\Bigl( \varepsilon^{\frac{1}{2}}\pi_i(\varepsilon\, y)
+ \sum_{j\not=i} w_{\mu_j^{\prime},\zeta_j^{\prime}}(y)+\varepsilon^{\frac{1}{2}}\pi_j(\varepsilon\, y)
\Bigr) \\
\label{expansionE}
& \quad + O( w_{\mu_i^{\prime},\zeta_i^{\prime}}(y)^3  \varepsilon^2) + O(\varepsilon^5)
,\quad \text{for } y \in B_{\delta/\varepsilon}(\zeta_i').
\end{align}
Now, Lemma~\ref{lemma22} guarantees that, for $y \in B_{\delta/\varepsilon}(\zeta_i')$,
\begin{align}
\label{pi}
\pi_i(\varepsilon y)
&= -4\pi \alpha_3  \mu_i^{\frac{1}{2}} H_\lambda(\varepsilon y,\zeta_i)
+O( \mu_i^{\frac{3}{2}} )
=-4\pi \alpha_3  \mu_i^{\frac{1}{2}} g_\lambda(\zeta_i)
+O( \varepsilon^{\frac{3}{2}} ) .
\end{align}
Similarly, Lemma~\ref{Uiexp} yields that, for $y \in B_{\delta/\varepsilon}(\zeta_i'))$ and $j\not=i$,
\begin{align}
\nonumber
w_{\mu_j^{\prime},\zeta_j^{\prime}}(y)+\varepsilon^{\frac{1}{2}}\pi_j(\varepsilon\, y)
&=  V_j(y) = \varepsilon^{ \frac{1}{2}} U_j( \varepsilon y)
\\
\nonumber
& = 4\pi\,\alpha_3 \varepsilon^{ \frac{1}{2}}  \mu_j^{\frac{1}{2}} \,G_{\lambda}(\varepsilon y ,\zeta_j)+O(\mu_i^{\frac{5}{2}-\sigma} )
\\
\label{vJ}
&=  4\pi\,\alpha_3\varepsilon^{ \frac{1}{2}}  \mu_j^{\frac{1}{2}}\,G_{\lambda}(\zeta_i ,\zeta_j) + O(\varepsilon^2).
\end{align}
Using \eqref{expansionE}, along with \eqref{pi} and \eqref{vJ}, we find that
\begin{align}
\nonumber
E(y)
& = 20\pi \alpha_3 \varepsilon^{\frac{1}{2}}
w_{\mu_i^{\prime},\zeta_i^{\prime}}(y)^4
\Bigl(
-\mu_i^{\frac{1}{2}} g_\lambda(\zeta_i)
+\sum_{j\not=i} \mu_j^{\frac{1}{2}} G_\lambda(\zeta_i,\zeta_j)
\Bigr) \\
\label{expansionE2}
& \quad + O( w_{\mu_i^{\prime},\zeta_i^{\prime}}(y)^3  \varepsilon^2) + O(\varepsilon^5) ,
\quad \text{for } y \in B_{\delta/\varepsilon}(\zeta_i')),
\end{align}
which implies
\[
\sup_{y \in B_{\delta/\varepsilon}(\zeta_i'))}
\omega(y)^{-(2+\nu)}   |E(y)|
\leq C \varepsilon^{\frac{1}{2}}
\bigl|
- \mu_i^{\frac{1}{2}} g_\lambda(\zeta_i)  + \sum_{j\not=i} \mu_j^{\frac{1}{2}} \,G_{\lambda}(\zeta_i ,\zeta_j)
\bigr| + C \varepsilon^2.
\]
This together with \eqref{EcomplementB} yields the desired estimate.
\end{proof}

We note that just assuming that $\mu_i$, $\zeta_i$ satisfy \eqref{parametros} we have $|M_\lambda(\zeta) \mu^{\frac{1}{2}}| \leq C \varepsilon^{\frac{1}{2}}$ and hence
\begin{align}
\label{estE1}
\| E \|_{**}\leq C \varepsilon.
\end{align}
However, this  estimate is not sufficient to prove the main theorem. An essential part of the argument is to work with $\zeta$ and $\mu$ so that  $M_\lambda(\zeta) \mu^{\frac{1}{2}} $  is smaller than $\varepsilon^{\frac{1}{2}}$.

\medskip

\begin{lemma}
\label{lema2}
Assume that $\zeta_i'$, $\mu_i'$ satisfy \eqref{parametros} where $\delta>0$ is fixed small.
Then there exist
$\varepsilon_1>0$, $C_1>0$, such that for all $\varepsilon \in (0,\varepsilon_1)$ problem \eqref{linearizedproblemproj2} has a unique solution $\phi$ that satisfies
\begin{align}
\label{estPhi}
\|\phi\|_*\leq C ( \varepsilon^{\frac{1}{2}} |M_\lambda(\zeta) \mu^{\frac{1}{2}}| + \varepsilon^2) .
\end{align}
\end{lemma}
\begin{proof}
In order to find a solution to problem \eqref{linearizedproblemproj2} it is sufficient to solve the fixed point problem
\[
\phi= A(\phi),
\]
where
\begin{align}
\label{def-A}
A(\phi) = - T(N(\phi)+E) ,
\end{align}
and  $T$ is the linear operator defined in Proposition \ref{solvability}.

Now, for a small $\gamma>0$, let us consider the ball
$\mathcal{F}_\gamma:= \{\phi \in C(\overline{\Omega}_{\varepsilon}) \, \vline \ \|\phi\|_*\leq \gamma  \}.$
We shall prove that
$A$ is a contraction in $ \mathcal{F}_\gamma$ for small $\varepsilon>0$.
From Proposition \ref{solvability}, we get
\[ \|A(\phi)\|_*\leq C\left[ \|N(\phi)\|_{**}+\|E\|_{**}   \right].\]
Writing the formula for $N$ as
\[
N(\phi)=20\int_{0}^1(1-t)\,[V+t\phi]^3\,dt \,\phi^2,
\]
we get the following estimates which are valid for
$\phi_1, \phi_2\in \mathcal{F}_\gamma$,
\[
\|N(\phi_1)\|_{**}\leq C \|\phi_1\|_*^2 ,
\]
\begin{equation}
\label{N}
\|N(\phi_1)-N(\phi_2)\|_{**}\leq C \,\gamma  \,\|\phi_1-\phi_2\|_*.
\end{equation}
Thus, we can deduce the existence of a constant $C>0$ such that
\[
\|A(\phi)\|_*\leq C \left[\gamma^2 + \|E\|_{**} \right] .
\]
From Lemma~\ref{lemma-error} we obtain the basic estimate $\|E\|_{**} \leq C\varepsilon$ with $C$ independent of the parameters $(\mu,\zeta)$ satisfying \eqref{parametros}.
Choosing $\gamma = 2 C \|E \|_{**} $ we see that $A$ maps $\mathcal{F}_\gamma$ into itself if $\gamma\leq \frac{1}{2C}$, which is true for $\varepsilon>0$ small.
Using now \eqref{N} we obtain
\[
\|A(\phi_1) - A(\phi_2) \|_* \leq C \gamma  \|\phi_1-\phi_2\|_*
\]
for $\phi_1,\phi_2 \in \mathcal{F}_\gamma$. Therefore $A$ is a contraction in $ \mathcal{F}_\gamma$ for small $\varepsilon>0$ and hence a unique fixed point of $A$ exists in this ball.
The solution $\phi$ satisfies
\begin{align}
\label{estPhiNonlinear}
\|\phi\|_* \leq \gamma = 2 C \|E\|_{**} \leq  C ( \varepsilon^{\frac{1}{2}} |M_\lambda(\zeta) \mu^{\frac{1}{2}}| + \varepsilon^2) ,
\end{align}
by Lemma~\ref{lemma-error}.
This concludes the proof of the lemma.
\end{proof}

\medskip

We shall next analyze the differentiability of the map $(\zeta',\mu')\rightarrow \phi$.

First we claim that:
\begin{lemma}
Assume that the parameters $\mu_i,\zeta_i$ satisfy \eqref{parametros}. Then
\begin{align}
\label{estDerEMu}
\| D_{\mu_i'} E \|_{**} \leq C \varepsilon ,
\\
\label{estDerEZeta}
\| D_{\zeta_i^\prime}\, E \|_{**} \leq C \varepsilon .\end{align}
\end{lemma}

\begin{proof}

First we observe that
\begin{align*}
\partial_{\mu_i^\prime}w_{\mu_i',\zeta_i'}&=
\frac{\alpha_3 \,\bigl(|y-\zeta_i'|^2-\mu_i'^2\bigr)}{2\,\sqrt{\mu_i'}\,\bigl(|y-\zeta_i'|^2+\mu_i'^2\bigr)^{\frac{3}{2}}}, \quad
D_{\zeta_i^\prime}w_{\mu_i',\zeta_i'}=
\frac{\alpha_3\, \sqrt{\mu_i'}\,(y-\zeta_i')}{\bigl(|y-\zeta_i'|^2+\mu_i'^2\bigr)^{\frac{3}{2}}} .
\end{align*}
and hence
\begin{equation}
\label{DMUZIE}
|\partial_{\mu_i^\prime}w_{\mu_i',\zeta_i'}|\leq C\, w_{\mu_i',\zeta_i'} \quad\text{and}\quad
|D_{\zeta_i^\prime}\,w_{\mu_i',\zeta_i'}|\leq C\, w_{\mu_i',\zeta_i'}^2.
\end{equation}
Let us prove \eqref{estDerEZeta}, the other being similar.
Let us assume without loss of generality that $i=1$.
Recall that
\[
E=V^5-
\sum_{i=1}^k   w_{\mu_i^{\prime},\zeta_i^{\prime}}^5,
\]
and so
\begin{align*}
D_{\zeta_1^\prime}\, E
&=
5V^4 D_{\zeta_1^\prime}\, V_1
- 5 w_{\mu_1',\zeta_1'}^4  D_{\zeta_1^\prime}\, w_{\mu_1',\zeta_1'}
\\
&=
5
\biggl[
\Bigl(
\sum_{i=1}^k w_{\mu_i',\zeta_i'} + \varphi_i
\Bigr)^4
-
w_{\mu_1',\zeta_1'}^4 \biggr]
D_{\zeta_1^\prime}\,w_{\mu_1',\zeta_1'}
 +
5\,\Bigl(
\sum_{i=1}^k w_{\mu_i',\zeta_i'} + \varphi_i
\Bigr)^4  \, D_{\zeta_1^\prime}\,  \varphi_1 ,
\end{align*}
where $\varphi_i(y) = \varepsilon^{1/2} \pi_i(\varepsilon y)$.
By \eqref{DMUZIE}, we have that, for $y \in B_{\delta/\varepsilon}(\zeta_1')$,
\begin{align}
\nonumber
&
\biggl|
\biggl(
\Bigl(
\sum_{i=1}^k w_{\mu_i',\zeta_i'} + \varphi_i
\Bigr)^4
-
w_{\mu_1',\zeta_1'}^4 \biggr)\,
D_{\zeta_1'} w_{\mu_1',\zeta_1'}
\biggr|
\\
\nonumber
& \leq
C w_{\mu_1',\zeta_1'}^3 \Bigl( | \varphi_1 |+ \sum_{i=2}^k \bigl( |w_{\mu_i',\zeta_i'}| + | \varphi_i| \bigr) \Bigr) \, |D_{\zeta_1'} w_{\mu_1',\zeta_1'} |
\\
\label{estDerE1}
& \leq C \,\varepsilon \, w_{\mu_1',\zeta_1'}^5  .
\end{align}
Note that from Lemma~\ref{lemma22}, $ |D_{\zeta_1'}  \varphi_1 (y)|\leq C \varepsilon^2$. Then,
 for  $y \in B_{\delta/\varepsilon}(\zeta_1')$,
\begin{align}
\nonumber
\left|
5 \Bigl(
\sum_{i=1}^k w_{\mu_i',\zeta_i'} + \varphi_i
\Bigr)^4   \,D_{\zeta_1'}  \varphi_1
\right|
& \leq C \left( w_{\mu_1',\zeta_1'}^4 + \varepsilon^4 \right) \varepsilon^2
\\
\label{estDerE2}
& \leq C \,\varepsilon^2\, w_{\mu_1',\zeta_1'}^4.
\end{align}
Using \eqref{estDerE1} and \eqref{estDerE2} we find that
\[
\sup_{y \in B_{\delta/\varepsilon}(\zeta_1')}
\omega(y)^{-(2+\nu)} |D_{\zeta_1'} E(y)| \leq C \varepsilon.
\]
The supremum on the rest of $\Omega_\varepsilon$ can be estimated similarly and this yields \eqref{estDerEZeta}.
\end{proof}


\begin{lemma}
Assume that $\zeta$, $\mu$ satisfy \eqref{parametros}.
Then
\begin{align}
\label{derPhiZeta}
\|  D_{\zeta_i'} \phi\|_* & \leq C (  \| E\|_{**} + \|D_{\zeta'} E\|_{**} ) ,
\\
\label{derPhiMu}
\|  D_{\mu_i'} \phi\|_*&  \leq C (  \| E\|_{**} + \|D_{\mu'} E\|_{**} ) .
\end{align}
\end{lemma}
\begin{proof}
To prove  differentiability of the function $\phi(\zeta')$ we first
recall that $\phi$ is found solving the fixed point problem
\[
\phi = A (\phi; \mu',\zeta')
\]
where $A$ is given in \eqref{def-A} but now we emphasize the dependence on $\mu ',\zeta '$.
Formally, differentiating this equation with respect to $\zeta_i'$ we find
\begin{align}
\label{fixedDerPhi}
D_{\zeta_i'} \phi =
\partial_{\zeta_i'} A(\phi;\mu',\zeta') +
\partial_\phi A(\phi;\mu',\zeta')
[D_{\zeta_i'} \phi].
\end{align}
The notation we are using is $D_{\zeta_i'}$ for the total derivative of the corresponding function and $\partial_{\zeta_i'}$ for the partial derivative.
From this fixed point problem for $D_{\zeta_i'} \phi $ we shall derive an estimate for $\|D_{\zeta_i'} \phi\|_*$.

Since $A(\phi;\mu',\zeta')  = - T( N(\phi;\mu',\zeta') + E;\mu',\zeta')$ we get
\begin{align*}
\partial_{\zeta_i'}A(\phi;\mu',\zeta')
&=
-\partial_{\zeta_i'}
T( N(\phi;\mu',\zeta') + E;\mu',\zeta')
-
T( \partial_{\zeta_i'}N(\phi;\mu',\zeta') ;\mu',\zeta')
\\
& \quad
- T(   D_{\zeta_i'} E;\mu',\zeta') .
\end{align*}
From Proposition \ref{solvability} we see that
\[
\| T(   D_{\zeta_i'} E;\mu',\zeta') \|_*
\leq
C \| D_{\zeta_i'} E\|_{**} .
\]
Using Proposition~\ref{propDerLinearOp} and estimates \eqref{estPhiNonlinear} and \eqref{N}, we find that
\[
\| \partial_{\zeta_i'}
T\bigl( N(\phi;\mu',\zeta') + E;\mu',\zeta'\bigr) \|_*
\leq C \|  N(\phi;\mu',\zeta') + E \|_{**}
\leq
C \|E \|_{**} .
\]
Similarly,
\begin{align*}
\| T( \partial_{\zeta_i'}N(\phi;\mu',\zeta') ;\mu',\zeta') \|_*
& \leq
C
\|  \partial_{\zeta_i'}N(\phi;\mu',\zeta') \|_{**}
\leq
C
\|  \phi  \|_{*}^2
\leq C  \|E\|_{**}^2.
\end{align*}
Therefore,
\begin{align}
\label{estDerAZeta}
\| D_{\zeta_i'}A(\phi;\mu',\zeta')  \|_* \leq C \|E\|_{**}.
\end{align}
Next we estimate
\begin{align}
\nonumber
\|
\partial_\phi A(\phi;\mu',\zeta')
[D_{\zeta_i'} \phi] \|_*
&=
\|
T ( \partial_\phi N(\phi;\mu',\zeta')
[D_{\zeta_i'} \phi] ) \|_*
\\
\nonumber
& \leq
\| \partial_\phi N(\phi;\mu',\zeta')
[D_{\zeta_i'} \phi]  \|_{**}
\\
\nonumber
& \leq
C \, \|\phi\|_* \, \|D_{\zeta_i'} \phi\|_*
\\
\label{estDerAPhi}
& \leq
C \, \| E \|_* \, \|D_{\zeta_i'} \phi\|_* .
\end{align}
From \eqref{estDerAZeta}, \eqref{estDerAPhi} and the fixed point problem \eqref{fixedDerPhi} we deduce \eqref{derPhiZeta}.
The proof of \eqref{derPhiMu} is similar.
\end{proof}

As a corollary of the previous lemma and taking into account
\eqref{estDerEMu}, \eqref{estDerEZeta}, and \eqref{estE1} we get the following estimate
\begin{align}
\label{estDerPhi2}
\|D_{\zeta_i'} \phi \|_* + \|D_{\mu_i'} \phi \|_*  \leq C \varepsilon .
\end{align}

\section{The reduced energy}
\label{secReduction}

After Problem (\ref{linearizedproblemproj}) has been solved, we will find a solution to the original problem (\ref{maineqreesc}) if we manage to adjust the pair $(\zeta',\mu' )$ in such a way that $c_{i}(\zeta',\mu' )=0$, $i=1,2,3,4$. This is the {\em reduced problem} and it turns out to be variational, that is, its solutions are critical points of the reduced energy functional
\begin{align}
\label{defIlambda}
I_\lambda(\zeta',\mu' )
=
\bar J_\lambda(V+\phi)
\end{align}
where $\bar J_\lambda$ is the energy functional for the problem \eqref{maineqreesc}, that is,
\[
\bar J_\lambda(v)=\frac{1}{2}\int_{\Omega_\varepsilon}\vert\nabla v\vert^2 - \varepsilon^2\, \frac{ \lambda}{2}\int_{\Omega_\varepsilon}v^2-\frac{1}{6}\int_{\Omega_\varepsilon}v^6,
\]
the function $V$ is the ansatz given in \eqref{defV} and $\phi = \phi(\zeta',\mu' )$ is the solution of (\ref{linearizedproblemproj}) constructed in Lemma~\ref{lema2} for $\varepsilon \in (0,\varepsilon_1)$.

\begin{lemma}
\label{lemmaReduction1}
Assume that $\zeta_i'$, $\mu_i'$ satisfy \eqref{parametros} where $\delta>0$ is fixed small and $\varepsilon_1>0$ is small as in Lemma~\ref{lema2}. Then $I_\lambda$ is $C^1$ and
$V+\phi$ is a solution to \eqref{maineqreesc} if and only if
\begin{align}
\label{reducedSystem}
D_{\zeta'} I_\lambda(\zeta',\mu' )=0,
\quad
D_{\mu'} I_\lambda(\zeta',\mu' )=0 .
\end{align}
\end{lemma}
\begin{proof}
Differentiating $I_\lambda$ with respect to $\mu_n'$
and using that $\phi$ solves \eqref{linearizedproblemproj}
we find
\begin{align*}
\partial_{\mu_n'}
I_\lambda(\zeta',\mu' )
&=
D \bar J_\lambda(V+\phi)[\partial_{\mu_n'} V + \partial_{\mu_n'} \phi]
\\
&= - \sum_{i,j} c_{ij} \int_{\Omega_\varepsilon}
w_{\mu_i^{\prime},\zeta_i^{\prime}}^4\,z_{ij}
\,( \partial_{\mu_n'} V + \partial_{\mu_n'} \phi ) .
\end{align*}
Similarly
\begin{align*}
D_{\zeta_{n}'}
I_\lambda(\zeta',\mu' )
&= -  \sum_{i,j} c_{ij} \int_{\Omega_\varepsilon}
w_{\mu_i^{\prime},\zeta_i^{\prime}}^4\,z_{ij} \,
( D_{\xi_{n}'} V + D_{\xi_{n}'} \phi ) .
\end{align*}
Since all terms in these expressions depends continuously on $\zeta',\mu'$ we deduce that $I_\lambda$ is $C^1$.

Clearly if $V+\phi$ is a solution to \eqref{maineqreesc} then all $c_{ij}=0$ and hence \eqref{reducedSystem} holds.
Reciprocally, if \eqref{reducedSystem} holds, then
\begin{align}
\label{systemC}
\left\{
\begin{aligned}
\sum_{i,j} c_{ij}
\int_{\Omega_\varepsilon}
w_{\mu_i^{\prime},\zeta_i^{\prime}}^4\,z_{ij} \,
( \partial_{\mu_n'} V + \partial_{\mu_n'} \phi )
&=0
\\
\sum_{i,j} c_{ij}
\int_{\Omega_\varepsilon}
w_{\mu_i^{\prime},\zeta_i^{\prime}}^4\,z_{ij} \,
( D_{\zeta_{n}'} V \cdot e_l + D_{\zeta_{n}'} \phi \cdot e_l)
&=0,
\end{aligned}
\right.
\end{align}
for all $n=1,\ldots,k$. Thanks to \eqref{estDerPhi2} we see that
\[
\int_{\Omega_\varepsilon}
w_{\mu_i^{\prime},\zeta_i^{\prime}}^4\,z_{ij}\,
\partial_{\mu_n'} \phi  \to0,
\quad
\int_{\Omega_\varepsilon}
w_{\mu_i^{\prime},\zeta_i^{\prime}}^4\,z_{ij} \,
D_{\zeta_{n}'}  \phi   \to0,
\]
as $\varepsilon\to 0$. Also, by \eqref{zij} and the expansion in Lemma~\ref{lemma22} we find that
\[
\int_{\Omega_\varepsilon}
w_{\mu_i^{\prime},\zeta_i^{\prime}}^4\,z_{ij}
\,\partial_{\mu_n'} V
=
\delta_{j4}\, \delta_{ik}
\int_{\R^3} w_{\mu',0}^4 (\partial_\mu w_{\mu',0})^2 + o(1)
\]
and
\[
\int_{\Omega_\varepsilon}
w_{\mu_i^{\prime},\zeta_i^{\prime}}^4\,z_{ij}
\,D_{\zeta_{n}'}  V \cdot e_l
=
\delta_{ik}\,
\delta_{jl}
\int_{\R^3}
w_{\mu',0}^4 (\nabla w_{\mu',0}\cdot e_1)^2
+o(1)
\]
as $\varepsilon\to0$, for some $\mu' \in (\delta,\frac{1}{\delta})$.

Therefore the system of equations \eqref{systemC} is invertible for the $c_{ij}$ when $\varepsilon>0$ is small, and hence $c_{ij}=0$ for all $i,j$.
\end{proof}

A nice feature of the system of equations \eqref{reducedSystem} is that it turns out to be equivalent to finding critical points of a functional of the pair $(\zeta',\mu')$ which is close, in appropriate sense, to the energy of $k$ bubbles $U_1 + \ldots +  U_k$.

\begin{lemma}
\label{lemmaApproxEnergy}
Assume the same conditions as in Lemma~\ref{lemmaReduction1}.
Then
\begin{align}
\label{expansion2}
I_\lambda(\zeta',\mu') = J_\lambda(\sum_{i=1}^k U_i) + \theta_\lambda^{(2)}(\zeta',\mu'),
\end{align}
where $\theta$ satisfies
\[
\theta_\lambda^{(2)}(\zeta',\mu')  =  - \int_0^1 s \left[\int_{\Omega_\varepsilon}\vert\nabla \phi\vert^2 - \varepsilon^2 \lambda \phi^2-5 (V + s\phi)^4 \phi^2\right]\,ds ,
\]
where $\phi = \phi(\zeta',\mu') $ is the solution of \eqref{linearizedproblemproj2} found in Lemma~\ref{lema2}.
\end{lemma}
\begin{proof}
From Taylor's formula
we find that
\begin{align*}
I_\lambda(\zeta',\mu' )
&=
\bar J_\lambda(V)
+ D \bar J_\lambda(V+\phi) [\phi]
+ \theta_\lambda^{(2)}(\zeta',\mu'),
\end{align*}
where
\begin{align}
\label{formulaR}
\theta_\lambda^{(2)}(\zeta',\mu')
&= -\int_0^1 s D^2 \bar J_\lambda(V + s\phi)[\phi^2] \,ds .
\end{align}
But since  $\phi$ satisfies \eqref{linearizedproblemproj}, we have that
\begin{align*}
D \bar J_\lambda(V+\phi) [\phi]
&=- \sum_{i,j} c_{ij} \int_{\Omega_\varepsilon} w_{\mu_i^{\prime},\zeta_i^{\prime}}^4\,z_{ij} \phi = 0 ,
\end{align*}
which implies \eqref{expansion2}.
\end{proof}

We remark that assuming  \eqref{parametros} we get
\[
\left|
\theta_\lambda^{(2)}(\zeta',\mu')  \right|\leq C \varepsilon^2 ,
\]
since \eqref{estPhi} holds.

\section{Critical multi-bubble}
\label{sectProof}
\noindent
Let $k\geq2$ be a given integer.
For $\delta>0$ fixed small we consider the sets
\begin{multline*}
\Omega_\delta^k:=\{
\zeta\equiv(\zeta_1,\ldots,\zeta_k)\in\Omega^k: \,\textrm{dist}(\zeta_i,\partial \Omega)>\delta, \vert \zeta_i-\zeta_j\vert>\delta,
 i=1,\ldots,k,\, j\neq i
\}
\end{multline*}
Recall that the main term in the expansion of $J_\lambda\Bigl(\sum_{i=1}^k U_i\Bigr)$ is the function
\begin{align*}
F_\lambda(\zeta,\mu)
&:=
\,k\,a_0
+a_1\sum_{i=1}^k\Bigl(\mu_i\,g_{\lambda}(\zeta_i)-\sum_{j\neq i}\mu_i^{1/2}\,\mu_j^{1/2}\,G_{\lambda}(\zeta_i,\zeta_j)\Bigr)+a_2\,\lambda\,\sum_{i=1}^k \mu_i^2
\\
&
-a_3\,\sum_{i=1}^k\Bigl(\mu_i\,g_{\lambda}(\zeta_i)-\sum_{j\neq i}\mu_i^{1/2}\mu_j^{1/2}\,G_{\lambda}(\zeta_i,\zeta_j)\Bigr)^2 ,
\end{align*}
where $\zeta\in\Omega_\delta^k$,
$\mu\equiv(\mu_1,\ldots,\mu_k)\in(\R^+)^k$
and the constants $a_i$ are given in \eqref{a0}--\eqref{a3}.

\noindent

\begin{proof}[Proof of Theorem~\ref{thm1}]

By Lemma~\ref{lemmaReduction1}, $v= V+\phi$ solves \eqref{maineqreesc} if the function $I_\lambda(\zeta',\mu')$ defined in \eqref{defIlambda} has a critical point.

In the sequel we will write also $I_\lambda(\zeta,\mu)$ for the same function but depending on $\zeta$, $\mu$, which we always assume satisfy the relation \eqref{muiPrime} with $\zeta'$, $\mu'$.

Using the  expansion of $J_\lambda\Bigl(\sum_{i=1}^kU_i\Bigr)$ given in Lemma~\ref{lemmaEnergyExpansion}, together with Lemma~\ref{lemmaApproxEnergy}, we see that
$I_\lambda(\zeta,\mu)$ has the form
\[
I_\lambda(\zeta,\mu) = F_\lambda(\zeta,\mu) + \theta_\lambda(\zeta,\mu)
\]
where $\theta_\lambda(\zeta,\mu)= \theta_\lambda^{(1)}(\zeta,\mu) + \theta_\lambda^{(2)}(\zeta,\mu) $,  $\theta_\lambda^{(1)}$ is the remainder that appears in Lemma~\ref{lemmaEnergyExpansion} and $\theta_\lambda^{(2)}$ the remainder in Lemma~\ref{lemmaApproxEnergy}.

It is convenient to perform the change of variables
\begin{align}
\label{LambdaMu}
\Lambda_i := \mu_i^{1/2} ,
\end{align}
where now $\Lambda\equiv(\Lambda_1,\ldots,\Lambda_k)\in \R^k$, and write, with some abuse of notation,
\begin{align*}
F_{\lambda}(\zeta,\Lambda)
:=&
\,k\,a_0
+a_1\sum_{i=1}^k\Bigl(\Lambda_i^2\,g_{\lambda}(\zeta_i)-\sum_{j\neq i}\Lambda_i\,\Lambda_j\,G_{\lambda}(\zeta_i,\zeta_j)\Bigr)+a_2\,\lambda\,\sum_{i=1}^k \Lambda_i^4\\
&-a_3\,\sum_{i=1}^k\Bigl(\Lambda_i^2\,g_{\lambda}(\zeta_i)-\sum_{j\neq i}\Lambda_i \Lambda_j\,G_{\lambda}(\zeta_i,\zeta_j)\Bigr)^2.
\end{align*}
Note that $\partial_{\mu'_i} I_\lambda(\mu',\zeta')=0$ is equivalent to  $\partial_{\mu_i} \tilde F_\lambda=0$, whenever $\Lambda_i \not=0$.

The function $F_\lambda$ can be expressed in terms of the matrix $M_\lambda$ as
\[F_{\lambda}(\zeta,\Lambda)
=
\,k\,a_0
+a_1\,
\Lambda^T M_\lambda(\zeta) \Lambda
+a_2\,\lambda\,\sum_{i=1}^k \Lambda_i^4-a_3\,\sum_{i=1}^k\Lambda_i^2\,( M_\lambda(\zeta) \Lambda )_i^2.
\]
In what follows we write $\sigma_1(\varepsilon,\zeta) $ for the smallest eigenvalue of $M_\lambda(\zeta)$ where $\lambda = \lambda_0 + \varepsilon$.
Using the Perron-Frobenius theorem or a direct argument as in \cite{bahri-li-rey} the eigenvalue $\sigma_1(\varepsilon,\zeta) $ is simple and has an eigenvector $v_1(\varepsilon,\zeta) $ with $|v_1(\varepsilon,\zeta) |=1$ and whose components are all positive.
By a standard application of the implicit function theorem, we have that $\sigma_1(\varepsilon,\zeta)$ and  $v_1(\varepsilon,\zeta)$ are smooth functions of $\varepsilon$ and $\zeta$ in a neighborhood of $(0,\zeta^0)$.

\noindent
We also have the following properties as a consequence of the hypothesis:
\[
D_\zeta \sigma_1(0,\zeta^0) = 0,\qquad
D^2_{\zeta\zeta} \sigma_1(0,\zeta^0)\, \text{is nonsingular},\qquad
\frac{\partial  \sigma_1 }{\partial \lambda} (0,\zeta^0)<0.
\]
These assertions can be proved by observing that
\[
\psi_{\lambda_0+\varepsilon}(\zeta)=\det M_{\lambda_0+\varepsilon}(\zeta) = \sigma_1(\varepsilon,\zeta) \sigma_*(\varepsilon,\zeta) ,
\]
where  $\sigma_*(\varepsilon,\zeta)$ is the product of the rest of the eigenvalues of $M_{\lambda_0+\varepsilon}(\zeta)$. Since $\sigma_1$ is a simple eigenvalue and $M_{\lambda_0}(\zeta^0)$ is positive semidefinite, we have $\sigma_*(0,\zeta^0)>0$ and this is still true for $\varepsilon,\zeta$ in a neighborhood of $(0,\zeta^0)$.
Then the properties stated above for $\sigma_1$ follow from  our assumptions on  $\psi_{\lambda_0+\varepsilon}(\zeta)$.

Since $\frac{\partial  \sigma_1 }{\partial \lambda} (0,\zeta^0)<0$,
we deduce that there are $\varepsilon_0>0$ and $c_0>0$ such that
 \begin{align}
\label{sigma-negative}
\sigma_1(\varepsilon,\zeta) <0 ,
\quad
\text{for } \varepsilon \in (0,\varepsilon_0),
\quad
\zeta \in B_{c_0 \sqrt{\varepsilon}}(\zeta^0).
\end{align}
Next we construct a $k\times k$ matrix  $P (\varepsilon,\zeta)$  for  $\varepsilon$ and $\zeta$ in a neighborhood of $(0,\zeta^0)$  with the following properties:
\begin{itemize}
\item[a)] the first column of $P$ is $v_1(\varepsilon,\zeta)$,
\item[b)] columns 2 to $k$ of $P$  are orthogonal to $v_1(\varepsilon,\zeta)$,
\item[c)] $P (\varepsilon,\zeta)$ is smooth  for  $\varepsilon$ and $\zeta$ in a neighborhood of $(0,\zeta^0)$,
\item[d)]  $P (0,\zeta^0)$ is such that $M_{\lambda_0}(\zeta^0)  = P (0,\zeta^0) D P (0,\zeta^0)^T$ with $D$ diagonal,
\item[e)] $P (0,\zeta^0)^T P (0,\zeta^0) = I$.
\end{itemize}
To achieve this we let $\bar v_1,\ldots,\bar v_k$ be an orthonormal basis of $\R^k$ of eigenvectors of $M_{\lambda_0}(\zeta^0)$ such that $\bar v_1 = v_1(0,\zeta^0)$.
We  let, for $\varepsilon>0$ and $\zeta$ close to $\zeta^0$,
\[
v_i(\varepsilon,\zeta) = \bar v_i - (\bar v_i\cdot v_1(\varepsilon,\zeta) ) v_1(\varepsilon,\zeta)
, \quad 2\leq i\leq k,
\]
and $P $ be the matrix whose columns are $v_1(\varepsilon,\zeta) ,\ldots, v_k(\varepsilon,\zeta) $.

We remark that although it would be more natural to consider a matrix $\tilde P(\varepsilon,\zeta)$, which diagonalizes $M_\lambda(\zeta)$, this matrix may not be differentiable with respect to $\varepsilon$ and $\zeta$. For this reason we choose to work with $P$ as defined before.

Let us perform the following change of variables
\begin{align}
\label{changeLambda}
\Lambda=|\sigma_1|^{1/2}P(\varepsilon,\zeta) \bar{\Lambda} .
\end{align}
Note that the quadratic form $\Lambda^T M_\lambda(\zeta) \Lambda$ can be written as
\begin{align*}
\Lambda^T M_\lambda(\zeta) \Lambda
= \sigma_1(\varepsilon,\zeta) |\sigma_1(\varepsilon,\zeta) | \bar\Lambda_1^2 + |\sigma_1(\varepsilon,\zeta)| (\bar\Lambda')^T  Q(\varepsilon,\zeta) \bar\Lambda',
\end{align*}
where
\[
\bar\Lambda' =
\left[
\begin{matrix}
\bar\Lambda_2 \\
\vdots\\
\bar\Lambda_k
\end{matrix}
\right]
,\quad
Q(\varepsilon,\zeta)=  P'(\varepsilon,\zeta)^T M_{\lambda_0+\varepsilon}(\zeta) P'(\varepsilon,\zeta)
\]
and $ P'(\varepsilon,\zeta) = [v_2,\ldots,v_k]$ is the matrix formed by the columns 2 to $k$ of $P(\varepsilon,\zeta)$.

Thus $I_\lambda(\zeta,\bar\Lambda) = F_\lambda(\zeta,\bar\Lambda) + \theta_\lambda(\zeta,\bar\Lambda)$ can be written as
\begin{align}
\nonumber
I_\lambda(\zeta,\bar{\Lambda}) & = ka_0
+ a_1 \Big[  -\sigma_1(\varepsilon,\zeta)^2 \bar\Lambda_1^2 + |\sigma_1(\varepsilon,\zeta)| (\bar\Lambda')^T  Q(\varepsilon,\zeta) \bar\Lambda' \Bigr]
\\
\label{formF1lambda}
& \quad
+\sigma_1(\varepsilon,\zeta)^2\,
\mathcal Poly_4(\varepsilon,\zeta,\bar \Lambda)
+ \theta_\lambda(\zeta,\bar \Lambda),
\end{align}
where
\begin{align*}
& \mathcal Poly_4(\varepsilon,\zeta,\bar \Lambda)
:=a_2\lambda\sum_{i=1}^k
\Bigl(\sum_{j=1}^k
P_{ij} (\varepsilon,\zeta) \bar{\Lambda}_j\Bigr)^4
\\
& \quad
-a_3\sum_{i=1}^k
\Big[
\Bigl(\sum_{j=1}^k
P_{ij}(\varepsilon,\zeta)\bar{\Lambda}_j\Bigr)^2
\Bigl(
\sigma_1 (\varepsilon,\zeta)v_{1,i}(\varepsilon,\zeta)  \bar\Lambda_1
+
\sum_{j=2}^k
\sum_{l=1}^k
(M_{\lambda_0+\varepsilon}(\zeta))_{il}\, P_{lj}(\varepsilon,\zeta) \,\bar{\Lambda}_j\Bigr)^2
\Big] ,
\end{align*}
and $\theta_\lambda(\zeta,\bar \Lambda)$ denotes the function $\theta_\lambda(\zeta,\mu) $ where we have used the transformations \eqref{LambdaMu} and \eqref{changeLambda}.

Note that $\mathcal Poly_4(\varepsilon,\zeta,\bar \Lambda) $ is a polynomial in the variables $\bar{\Lambda}_1,\ldots,\bar{\Lambda}_k$ of degree $4$ whose coefficients are functions of $\varepsilon$ and $\zeta$.

We need to solve the equations $D_{\zeta} I_{\lambda}=0$,
$\frac{\partial I_{\lambda}}{\partial\bar{\Lambda}_1}=0$, $\ldots,$
$\frac{\partial I_{\lambda}}{\partial\bar{\Lambda}_k}=0$.
Because of the the absolute value of $\sigma_1$ appearing in \eqref{formF1lambda} it is a bit more convenient to modify this function by defining
\begin{align*}
\bar F_{\lambda}(\zeta,\bar{\Lambda})
& = ka_0
- a_1  \sigma_1(\varepsilon,\zeta)^2 \bar\Lambda_1^2
- a_1 \sigma_1(\varepsilon,\zeta) (\bar\Lambda')^T  Q(\varepsilon,\zeta) \bar\Lambda'
\\
& \quad
+\sigma_1(\varepsilon,\zeta)^2\,  \mathcal Poly_4(\varepsilon,\zeta,\bar \Lambda)
+ \theta_\lambda(\zeta,\bar \Lambda) ,
\end{align*}
which coincides with $I_{\lambda}$ when $\sigma_1<0$.

%
%


Next we compute
\begin{align*}
D_{\zeta} \bar F_{\lambda}
& =
-2a_1\sigma_1\,(D_{\zeta}\sigma_1)\bar{\Lambda}_1^2
-a_1  (D_\zeta \sigma_1) (\bar\Lambda')^T   Q  \bar\Lambda'
- a_1 \sigma_1 (\bar\Lambda')^T  ( D_\zeta  Q) \bar\Lambda'
\\
& \quad
+2\,\sigma_1\,(D_{\zeta}\sigma_1) \, \mathcal Poly_4
+\sigma_1^2\, D_{\zeta} \mathcal Poly_4
+D_\zeta \theta_\lambda,
\\
\frac{\partial \bar F_\lambda}{\partial \bar\Lambda_1}
&=
- 2 a_1  \sigma_1^2 \bar\Lambda_1
+\sigma_1^2\,
\frac{\partial }{\partial \bar \Lambda_1}\mathcal Poly_4
+ \frac{\partial \theta_\lambda}{\partial \bar \Lambda_1}  ,
\\
\frac{\partial \bar F_\lambda}{\partial \bar\Lambda_l}
&=
- 2a_1 \sigma_1 \sum_{j=2}^k Q_{j-1,l-1} \bar\Lambda_j
+\sigma_1^2\,
\frac{\partial }{\partial \bar \Lambda_l} \mathcal Poly_4
+ \frac{\partial \theta_\lambda}{\partial \bar \Lambda_l}  ,
\end{align*}
with $ l=2,\ldots,k$.

Observe that, whenever $\sigma_1<0$, the equations
$D_{\zeta} \bar F_{\lambda}=0$,
$\frac{\partial \bar F_{\lambda}}{\partial\bar{\Lambda}_n}=0$, $n=1,\ldots,k$, are equivalent to
\begin{align}
\nonumber
0&=
-2a_1\bar{\Lambda}_1^2 (D_\zeta\sigma_1)
- \frac{a_1}{\sigma_1}  (D_\zeta \sigma_1 ) (\bar\Lambda')^T   Q  \bar\Lambda'
- a_1 (\bar\Lambda')^T  ( D_\zeta  Q ) \bar\Lambda'
\\
\label{eq1}
& \quad
+2 (D_{\zeta}\sigma_1) \, \mathcal Poly_4
+\sigma_1\, D_{\zeta}\mathcal Poly_4
+\frac{1}{\sigma_1}	 D_\zeta \theta_\lambda,
\\
\label{eq2}
0&=
- 2 a_1 \bar\Lambda_1
+\frac{\partial }{\partial \bar \Lambda_1}\mathcal Poly_4
+ \frac{1}{\sigma_1^2}\frac{\partial \theta_\lambda}{\partial \bar \Lambda_1} ,
\\
\label{eq3}
0&=- 2a_1  \sum_{j=2}^k Q_{j-1,l-1}  \bar\Lambda_j
+\sigma_1\,  \frac{\partial }{\partial \bar \Lambda_l}\mathcal Poly_4
+ \frac{1}{\sigma_1} \frac{\partial \theta_\lambda}{\partial \bar \Lambda_l}  ,
\end{align}
with $ l=2,\ldots,k$.
Note that we have normalized the equations (the first one was divided by $\sigma_1$, the second by $\sigma_1^2$ and the last ones by $\sigma_1$).

We claim that there exists $\varepsilon_0>0$ such that for each $\varepsilon\in (0,\varepsilon_0)$ the system \eqref{eq1}, \eqref{eq2}, \eqref{eq3}
has a solution $(\zeta(\varepsilon)$, $\bar\Lambda(\varepsilon))$ such that $\sigma_1(\varepsilon,\zeta(\varepsilon))<0$, thus yielding a critical point of $I_{\lambda_0+\varepsilon}$.

We will prove that \eqref{eq1}, \eqref{eq2}, \eqref{eq3} has a solution using degree theory in a ball  centered at a suitable point $(\bar\Lambda^0, \zeta^0)$ and with a conveniently small radius.

To find the center of this ball, let us consider a simplified version of equations \eqref{eq2}, \eqref{eq3}, by omitting the terms involving $\theta_\lambda$ and evaluating at $\varepsilon=0$, $\zeta=\zeta^0$.
Using that $Q(0,\zeta^0)$ is the diagonal matrix with entries $\sigma_2,\ldots,\sigma_k$,
where $0,\sigma_2,\ldots,\sigma_k$ are the eigenvalues of $M_{\lambda_0}(\zeta^0)$, we get
\begin{align}
\label{eq2a}
0&=
- 2 a_1 \bar\Lambda_1
+\frac{\partial }{\partial \bar \Lambda_1} \mathcal Poly_4(0,\zeta_0,\bar \Lambda) , \\
\label{eq3a}
0&=- 2a_1 \sigma_l \bar\Lambda_l
 ,\quad l=2,\ldots,k .
\end{align}
We note that there is a  solution of \eqref{eq2a}, \eqref{eq3a} which has the form
$\bar\Lambda^0= ( \bar\Lambda^0_1,\ldots,\bar\Lambda^0_k)$ with
\[
\bar\Lambda^0_l = 0 \qquad \text{for all }l =2,\ldots, k
\]
and
\begin{align}
\label{barLambda10}
\bar{\Lambda}^0_1:=\sqrt{\frac{a_1}{2 a_2 \lambda_0 \sum_{i=1}^k P_{i1}(0,\zeta^0)^4 }}  .
\end{align}
For later purposes it will be useful to know that the linearization of the functions on the right hand side of \eqref{eq2a} and \eqref{eq3a} around $\bar\Lambda^0$ define an invertible operator.
Since the right hand side of \eqref{eq3a} is a constant times the identity it is sufficient to study
the expression
$ - 2 a_1 \bar\Lambda_1 +\frac{\partial }{\partial \bar \Lambda_1}\mathcal Poly_4(0,\zeta_0,\bar \Lambda) $.
A straightforward computation yields
\begin{align}
\nonumber
\frac{\partial }{\partial\bar{\Lambda}_1}
\left[
 - 2 a_1 \bar\Lambda_1 + \frac{\partial }{\partial \bar \Lambda_1}\mathcal Poly_4 (0,\zeta_0,\bar \Lambda)
\right]
(\bar{\Lambda}^0)
&=-2a_1\,+
12a_2\lambda_0\Bigl(\sum_{i=1}^kP_{i1}(0,\zeta^0)^4\Bigr)(\bar{\Lambda}_1^0)^2
\\
\label{firstCoefficient}
& =4a_1,
\end{align}
which is nonzero.

We now introduce one more change of variables
\[
\widehat\Lambda_j = \bar\Lambda_j - \bar\Lambda_j^0, \quad 1\leq j\leq k .
\]
Define
\[
\Upsilon (\zeta,\widehat \Lambda) = A (\zeta,\widehat{\Lambda}) + \mathcal R  (\zeta,\widehat{\Lambda}),
\]
where
\[
A (\zeta,\widehat{\Lambda})  = (A_0 (\zeta,\widehat{\Lambda}) ,A_1 (\zeta,\widehat{\Lambda}) ,\ldots, A_k (\zeta,\widehat{\Lambda}) ))
\]
with
\begin{align*}
A_0 (\zeta,\widehat{\Lambda}) &=
- a_1 (\bar\Lambda_1^0)^2 D^2_{\zeta\zeta}\sigma_1(0,\zeta^0) (\zeta-\zeta^0),
\\
A_1 (\zeta,\widehat{\Lambda}) &= 4 a_1 \widehat\Lambda_1
+ \sum_{j=2}^k \frac{\partial^2 }{\partial \bar \Lambda_j \partial 	\bar \Lambda_1}
\mathcal Poly_4
(0,\zeta^0,\bar \Lambda^0)\widehat \Lambda_j
+D_\zeta \frac{\partial }{ \partial 	\bar \Lambda_1}
\mathcal Poly_4
(0,\zeta^0,\bar \Lambda^0)(\zeta-\zeta^0),
\\
A_l (\zeta,\widehat{\Lambda}) &= -2 a_1 \sigma_l \widehat \Lambda_l , \quad l=2,\ldots,k,
\end{align*}
and
\[
\mathcal{R} (\zeta,\widehat{\Lambda})  = (\mathcal{R}_0 (\zeta,\widehat{\Lambda}) ,\mathcal{R}_1 (\zeta,\widehat{\Lambda}) ,\ldots, \mathcal{R}_k (\zeta,\widehat{\Lambda}) ))
\]
with
\begin{align*}
\mathcal R_0 (\zeta,\widehat \Lambda)
&=
-a_1 (\bar \Lambda_1^0)^2 \bigl( D_\zeta\sigma_1(\varepsilon,\zeta) - D^2_{\zeta\zeta} \sigma_1(0,\zeta^0)(\zeta-\zeta^0) \bigr)
\\
& \quad
-2a_1 (2 \bar \Lambda_1^0  \widehat \Lambda_1 +  \widehat \Lambda_1^2) D_\zeta\sigma	_1(\varepsilon,\zeta)
\\
& \quad
-  \frac{a_1}{\sigma_1} D_\zeta \sigma_1 (\widehat \Lambda')^T   Q(\varepsilon,\zeta)  \widehat\Lambda'
- a_1 (\widehat\Lambda')^T  ( D_\zeta  Q(\varepsilon,\zeta) ) \widehat\Lambda'
\\
& \quad
+2 (D_{\zeta}\sigma_1) \, ( \mathcal Poly_4(\varepsilon,\zeta,\bar\Lambda^0 +  \widehat \Lambda)  -
\mathcal Poly_4(0,\zeta^0,\bar\Lambda^0 ) )
+\sigma_1\, D_{\zeta}\mathcal Poly_4(\varepsilon,\zeta,\bar\Lambda^0 +  \widehat \Lambda)
\\
&\quad
+\frac{1}{\sigma_1}	 D_\zeta \theta_\lambda(\zeta,\bar\Lambda^0 +  \widehat \Lambda),
\\
\mathcal R_1 (\zeta,\widehat \Lambda)
&=
\frac{\partial }{\partial \bar \Lambda_1}\mathcal Poly_4 (\varepsilon,\zeta,\bar\Lambda^0 +  \widehat \Lambda)
-
\frac{\partial }{\partial \bar \Lambda_1}\mathcal Poly_4 (0,\zeta^0,\bar\Lambda^0)
-   \sum_{j=1}^k \frac{\partial^2 }{\partial \bar \Lambda_j \partial 	\bar \Lambda_1}\mathcal Poly_4(0,\zeta^0,\bar \Lambda^0)\widehat \Lambda_j
\\
& \quad
-D_\zeta \frac{\partial }{ \partial 	\bar \Lambda_1}
\mathcal Poly_4
(0,\zeta^0,\bar \Lambda^0)(\zeta-\zeta^0)
+ \frac{1}{\sigma_1^2}\frac{\partial \theta_\lambda}{\partial \bar \Lambda_1} (\zeta,\bar\Lambda^0 +  \widehat \Lambda)  ,
\\
\mathcal R_l (\zeta,\widehat \Lambda)
&=- 2a_1  \sum_{j=2}^k ( Q_{j-1,l-1} (\varepsilon,\zeta) -\delta_{jl}) \widehat \Lambda_j
+\sigma_1(\varepsilon,\zeta)\,  \frac{\partial }{\partial \bar \Lambda_l} \mathcal Poly_4 (\varepsilon,\zeta,\Lambda^0 +  \widehat \Lambda)
\\
&\quad
+ \frac{1}{\sigma_1(\varepsilon,\zeta)} \frac{\partial \theta_\lambda}{\partial \bar \Lambda_l} (\zeta,\Lambda^0 +  \widehat \Lambda)  ,\quad l=2,\ldots k.
\end{align*}
Let us indicate the motivation for the definition of  $A_0$. In equation \eqref{eq1} we combine the terms $-2a_1\bar{\Lambda}_1^2 (D_\zeta\sigma_1) $ and $ 2 (D_{\zeta}\sigma_1) \, \mathcal Poly_4$ into the expression
\begin{align*}
&
-2a_1\bar{\Lambda}_1^2 (D_\zeta\sigma_1) + 2 (D_{\zeta}\sigma_1) \, \mathcal Poly_4
\\
&=
-2a_1 (\bar{\Lambda}_1^0)^2(D_\zeta\sigma_1)
-2a_1\bigl ( 2 \bar{\Lambda}_1^0 \widehat{\Lambda}_1 + \widehat{\Lambda}_1^2 \Bigr) (D_\zeta\sigma_1)
\\
& \quad
+ 2 (D_{\zeta}\sigma_1) \, \mathcal Poly_4(0,\zeta^0,\bar\Lambda_0)
+  2 (D_{\zeta}\sigma_1) \, ( \mathcal Poly_4 - \mathcal Poly_4(0,\zeta^0,\bar\Lambda_0) ) .
\end{align*}
In this expression we combine
\begin{align*}
-2a_1 (\bar{\Lambda}_1^0)^2(D_\zeta\sigma_1)
+ 2 (D_{\zeta}\sigma_1) \, \mathcal Poly_4(0,\zeta^0,\bar\Lambda_0)
&= 2  (D_{\zeta}\sigma_1)  \Bigl[
-a_1 (\bar{\Lambda}_1^0)^2 + \mathcal Poly_4(0,\zeta^0,\bar\Lambda_0) \Bigr].
\end{align*}
But  an explicit computation using \eqref{barLambda10} gives
\[
-a_1( \bar{\Lambda}_1^0)^2  + \mathcal Poly_4(0,\zeta^0,\bar\Lambda^0) =
-\frac{1}{2} a_1( \bar{\Lambda}_1^0)^2 .
\]
Then
\begin{align*}
&
-2a_1 (\bar{\Lambda}_1^0)^2(D_\zeta\sigma_1)
+ 2 (D_{\zeta}\sigma_1) \, \mathcal Poly_4(0,\zeta^0,\bar\Lambda_0)
\\
&= -a_1 (\bar \Lambda_1^0)^2 (D_\zeta\sigma_1)
\\
&=
-a_1 (\bar \Lambda_1^0)^2 D^2_{\zeta\zeta}\sigma_1(0,\zeta^0)(\zeta-\zeta^0)
-a_1 (\bar \Lambda_1^0)^2
\bigl(
(D_\zeta\sigma_1) - D^2_{\zeta\zeta}\sigma_1(0,\zeta^0)(\zeta-\zeta^0)
\bigr) .
\end{align*}
We define $A_0$ as  $-a_1 (\bar \Lambda_1^0)^2 D^2_{\zeta\zeta}\sigma_1(0,\zeta^0)(\zeta-\zeta^0)$ and we leave all the others terms in  $\mathcal R_0$.

Then the equations \eqref{eq1}, \eqref{eq2} and \eqref{eq3}  for the unknowns $\widehat\Lambda_j$, $1\leq j\leq k$ and $\zeta$ are equivalent to
\[
\Upsilon (\zeta,\widehat \Lambda)  = 0 .
\]
We are going to show that the this equation has a solution in the ball
\[
\mathcal B = \{ (\zeta,\widehat\Lambda) \in \R^{3k}\times \R^k : | (\zeta-\zeta^0,\widehat\Lambda) | < \varepsilon^{1-\sigma} \}
\]
with a fixed and small $\sigma>0$, using degree theory.



The linear operator $(\zeta,\widehat\Lambda) \mapsto A(\zeta-\zeta^0,\widehat{\Lambda})$ is invertible thanks to hypothesis (iii) in the statement of the theorem and \eqref{firstCoefficient}. Hence there is a constant $c>0$ such that
\[
|A(\zeta, \widehat\Lambda)|\geq c |( (\zeta-\zeta^0), \widehat\Lambda)|,
\]
for $(\zeta,\widehat{\Lambda} ) \in \partial \mathcal B$, if we take $\varepsilon>0$ sufficiently small.
To conclude that the equation $A(\zeta,\widehat{\Lambda}) + \mathcal R(\zeta,\widehat{\Lambda}) =0$ has a solution in $\mathcal B$, it suffices to verify that
\[
|\mathcal R(\zeta,\widehat{\Lambda})| \leq
o( \varepsilon^{1-\sigma})
\]
uniformly for  $( \zeta,\widehat{\Lambda} ) \in \bar{\mathcal B}$ as $\varepsilon\to 0$.

Before performing the computations we recall the assumptions we are imposing on $\mu$, $\zeta$.
From  \eqref{LambdaMu} and \eqref{changeLambda} we have
\[
\mu^{\frac{1}{2}} = |\sigma_1(\varepsilon,\zeta) |^{\frac{1}{2}} P(\varepsilon,\zeta) \bar{\Lambda}.
\]
Then for $( \zeta,\widehat{\Lambda} ) \in \bar{\mathcal B}$,
\begin{align}
\label{cotaZetaGorro}
|\zeta-\zeta^0|\leq  \varepsilon^{1-\sigma}
\end{align}
and
\begin{align}
\label{cotaLambdaGorro}
\bar\Lambda = \bar\Lambda^0 + \widehat{\Lambda},
\quad
| \widehat{\Lambda} | \leq \varepsilon^{1-\sigma}  .
\end{align}
Using Taylor's theorem we see that, for $|\zeta-\zeta^0|\leq \varepsilon^{1-\sigma}$,
\begin{align}
\label{cotasSigma}
-c_1\varepsilon\leq \sigma_1(\varepsilon,\zeta) \leq -c_2\varepsilon
\end{align}
with $c_1,c_2>0$, and in particular
\[
|\mu_i | \leq C \varepsilon ,\quad i=1,\ldots,k.
\]
Also for $|\zeta-\zeta^0|\leq \varepsilon^{1-\sigma}$,
\begin{align}
\label{cotaGradSigma1}
|D_\zeta  \sigma_1(\varepsilon,\zeta)| \leq C \varepsilon^{1-\sigma} .
\end{align}
We will also need  the following estimates: for   $( \zeta,\widehat{\Lambda} ) \in \bar{\mathcal B}$ we have
\begin{align}
\label{cota1}
|D_\zeta \theta_\lambda(\zeta, \bar\Lambda^0+\hat \Lambda)| & \leq C  \varepsilon^{3-\sigma},
\\
\label{cota2}
\Bigl|
\frac{\partial \theta_\lambda}{\partial \bar \Lambda_1} (\zeta,\bar\Lambda^0 +  \widehat \Lambda) \Bigr| & \leq C \varepsilon^{3-\sigma/2},
\\
\label{cota3}
\Bigl|
\frac{\partial \theta_\lambda}{\partial \bar \Lambda_l} (\zeta,\bar\Lambda^0 +  \widehat \Lambda) \Bigr| & \leq C \varepsilon^{3-\sigma} , \quad l=2,\ldots,k.
\end{align}
We will prove these estimates later on.

For   $( \zeta,\widehat{\Lambda} ) \in \bar{\mathcal B}$ let us estimate $\mathcal R_0 (\zeta,\widehat \Lambda) $.
We start with
\begin{align*}
& \big|
(\bar \Lambda_1^0)^2 \bigl( D_\zeta\sigma_1(\varepsilon,\zeta) - D^2_{\zeta\zeta} \sigma_1(0,\zeta^0)(\zeta-\zeta^0) \bigr)
\bigr|
\\
& \leq C |D_\zeta\sigma_1(\varepsilon,\zeta)-D_\zeta\sigma_1(0,\zeta)|
+C\bigl| D_\zeta\sigma_1(0,\zeta) - D_\zeta\sigma_1(0,\zeta^0) - D^2_{\zeta\zeta} \sigma_1(0,\zeta^0)(\zeta-\zeta^0) \bigr|
\\
& \leq C \varepsilon + C \varepsilon^{2-2\sigma} \leq C \varepsilon,
\end{align*}
since $|\zeta-\zeta^0|\leq \varepsilon^{1-\sigma}$.
Next,
\begin{align*}
\bigl|
(2 \bar \Lambda_1^0  \widehat \Lambda_1 +  \widehat \Lambda_1^2) D_\zeta\sigma	_1(\varepsilon,\zeta)
\bigr|
& \leq C\varepsilon^{2-2\sigma}
\end{align*}
because $|\widehat{\Lambda}_1|\leq \varepsilon^{1-\sigma}$ and $D_\zeta \sigma_1(0,\zeta^0)=0$.
To estimate $ \frac{D_\zeta \sigma_1}{\sigma_1}  (\widehat \Lambda')^T   Q(\varepsilon,\zeta)  \widehat\Lambda' $ we note that  \eqref{cotasSigma} together with \eqref{cotaGradSigma1} implies
\[
\Bigl|
\frac{D_\zeta \sigma_1(\varepsilon,\zeta)}{\sigma_1(\varepsilon,\zeta)}
\Bigr| \leq C \varepsilon^{-\sigma}
\]
and so
\begin{align*}
\Bigl|
a_1  \frac{D_\zeta \sigma_1}{\sigma_1}  (\widehat \Lambda')^T   Q(\varepsilon,\zeta)  \widehat\Lambda' \Bigr|
\leq C \varepsilon^{2-3\sigma}.
\end{align*}
Next, using \eqref{cota1} we estimate
\begin{align*}
\bigl|
a_1 (\widehat\Lambda')^T  ( D_\zeta  Q(\varepsilon,\zeta) ) \widehat\Lambda'
\bigr|
&\leq C \varepsilon^{2-2\sigma},
\\
\bigl|
2 (D_{\zeta}\sigma_1) \, ( \mathcal Poly_4(\varepsilon,\zeta,\bar\Lambda^0 +  \widehat \Lambda)  -
\mathcal Poly_4(0,\zeta^0,\bar\Lambda^0 ) )
\bigr|
&\leq C \varepsilon^{2-2\sigma} ,
\\
|\sigma_1(\varepsilon,\zeta)  D_\zeta \mathcal Poly_4( \varepsilon,\zeta,\bar\Lambda^0 + \widehat{\Lambda})|
& \leq C \varepsilon ,
\\
\left|\frac{1}{\sigma_1}	 D_\zeta \theta_\lambda(\zeta,\bar\Lambda^0 +  \widehat \Lambda)
\right|
&\leq C \varepsilon^{2-\sigma}
\end{align*}
This proves that
\begin{align}
\label{R0}
|\mathcal R_0 (\zeta,\widehat \Lambda) |\leq C \varepsilon
\end{align}
for   $( \zeta,\widehat{\Lambda} ) \in \bar{\mathcal B}$, if we have fixed $\sigma>0$ small.

Let us estimate $ |R_1 (\zeta,\widehat \Lambda) | $ for    $( \zeta,\widehat{\Lambda} ) \in \bar{\mathcal B}$.
By Taylor's theorem we have that
\begin{align*}
& \Bigl|\frac{\partial }{\partial \bar \Lambda_1}\mathcal Poly_4 (\varepsilon,\zeta,\bar\Lambda^0 +  \widehat \Lambda)
-
\frac{\partial }{\partial \bar \Lambda_1}\mathcal Poly_4 (0,\zeta^0,\bar\Lambda^0)
-   \sum_{j=1}^k \frac{\partial^2 }{\partial \bar \Lambda_j \partial 	\bar \Lambda_1}\mathcal Poly_4(0,\zeta^0,\bar \Lambda^0)\widehat \Lambda_j
\\
& \quad
-D_\zeta \frac{\partial }{ \partial 	\bar \Lambda_1}
\mathcal Poly_4
(0,\zeta^0,\bar \Lambda^0)(\zeta-\zeta^0)
\Bigr| \leq C \varepsilon + C |\zeta-\zeta^0|^2 + C |\widehat{\Lambda}|^2 \leq C \varepsilon.
\end{align*}
On the other hand by \eqref{cota2} we have
\begin{align*}
\left|
\frac{1}{\sigma_1^2}\frac{\partial \theta_\lambda}{\partial \bar \Lambda_1} (\zeta,\bar\Lambda^0 +  \widehat \Lambda)
\right| \leq C \varepsilon^{1-\sigma/2} .
\end{align*}
This shows that
\begin{align}
\label{R1}
|\mathcal R_1 (\zeta,\widehat \Lambda) |\leq C \varepsilon^{1-\sigma/2}
\end{align}
for   $( \zeta,\widehat{\Lambda} ) \in \bar{\mathcal B}$.

Finally, using \eqref{cota3}, we have that for    $( \zeta,\widehat{\Lambda} ) \in \bar{\mathcal B}$ and $l=2,\ldots,k$, the following holds
\begin{align*}
\biggl|
2a_1  \sum_{j=2}^k ( Q_{j-1,l-1} (\varepsilon,\zeta) -\delta_{jl}) \widehat \Lambda_j
\biggr|
\leq C \varepsilon^{2-2\sigma},
\end{align*}
\begin{align*}
\biggl| \sigma_1(\varepsilon,\zeta)\,  \frac{\partial }{\partial \bar \Lambda_l}  \mathcal Poly_4 (\varepsilon,\zeta,\Lambda^0 +  \widehat \Lambda)
\biggr|
\leq C \varepsilon^{2-2\sigma},
\end{align*}
and
\begin{align*}
\left|
\frac{1}{\sigma_1(\varepsilon,\zeta)} \frac{\partial \theta_\lambda}{\partial \bar \Lambda_l} (\zeta,\Lambda^0 +  \widehat \Lambda)
\right|
\leq C \varepsilon^{2-\sigma}.
\end{align*}
Therefore,
\begin{align}
\label{R2}
|\mathcal R_l (\zeta,\widehat \Lambda) |\leq C \varepsilon^{2-2\sigma}, \quad l=2,\ldots,k
\end{align}
for   $( \zeta,\widehat{\Lambda} ) \in \bar{\mathcal B}$.

Combining \eqref{R0}, \eqref{R1} and \eqref{R2} we obtain
\[
|\mathcal{R}(\zeta,\widehat{\Lambda})|\leq C \varepsilon^{1-\sigma/2},
\quad \forall ( \zeta,\widehat{\Lambda} ) \in \bar{\mathcal B}.
\]
A standard application of degree theory then yields a solution of $\Upsilon(\zeta,\widehat{\Lambda})=0$ in the ball $\mathcal{B}$.
Note that for $(\zeta,\widehat{\Lambda}) \in \mathcal B$ we are in the region where  \eqref{sigma-negative} holds, and hence $\sigma_1(\zeta , \bar\Lambda^0 + \widehat{\Lambda})<0$. Therefore we have found a critical point of $I_\lambda(\zeta,\mu)$, which was the desired conclusion.
\end{proof}

\begin{proof}[Proof of \eqref{cota1}, \eqref{cota2}, \eqref{cota3}]
By Lemma~\ref{lemmaEnergyExpansion} (using the satement with $\frac{\sigma}{2}$ instead of $\sigma$) we get directly the estimates
\begin{align}
\label{theta1a}
| D_\zeta  \theta_\lambda^{(1)} (\zeta,\mu)  | & \leq C |\mu|^{3-\sigma/2} \leq C \varepsilon^{3-\sigma/2} ,
\\
\label{theta1b}
\Bigl|
\frac{\partial \theta_\lambda^{(1)}}{\partial \bar\Lambda_i} (\zeta,\mu)  \Bigr|
&\leq C |\mu|^{3-\sigma/2} \leq C \varepsilon^{3-\sigma/2} .
\end{align}
To estimate $D_\zeta \theta_\lambda^{(2)}$, we recall formula \eqref{formulaR} which gives
\begin{align*}
\theta_\lambda^{(2)}(\zeta,\mu)
&= \int_0^1 s D^2 \bar J_\lambda(V + s\phi)[\phi^2] \,ds .
\\
&=
\int_0^1 s \left[\int_{\Omega_\varepsilon}\vert\nabla \phi\vert^2 - \varepsilon^2 \lambda \phi^2-5 (V + s\phi)^4 \phi^2\right]\,ds
\end{align*}
and therefore
\begin{align*}
|D_\zeta \theta_\lambda^{(2)}(\zeta,\mu) |
\leq C \|\phi\|_{*} \|D_\zeta\phi\|_* + \frac{C}{\varepsilon}\|\phi\|_*^2.
\end{align*}
We can compute
\begin{align*}
M_\lambda(\zeta) \mu^{\frac{1}{2}}
&=|\sigma_1|^{\frac{1}{2}} M_\lambda P \bar\Lambda
=|\sigma_1|^{\frac{1}{2}} \Bigl( \sigma_1 v_1 \bar\Lambda_1 + \sum_{l=2}^k \bar v_l \bar  \Lambda_l \Bigr) ,
\end{align*}
and thanks to \eqref{cotasSigma} we see that
\[
|M_\lambda(\zeta) \mu^{\frac{1}{2}}|\leq C \varepsilon^{\frac{3}{2}-\sigma} ,
\]
which in turn implies
\[
\| E \|_{**} \leq C \varepsilon^{2-\sigma}, \quad
\| \phi \|_{*} \leq C \varepsilon^{2-\sigma} .
\]
From this we deduce
\[
|  \theta_\lambda^{(2)}(\zeta,\mu)   |\leq C \varepsilon^{4-2\sigma}.
\]
We can write \eqref{expansionE2} in the form (near $\zeta_i'$)
\begin{align*}
E(y)
& = - 20\pi \alpha_3 \varepsilon^{\frac{1}{2}}
w_{\mu_i^{\prime},\zeta_i^{\prime}}(y)
M_\lambda \mu^{\frac{1}{2}} + O( w_{\mu_i^{\prime},\zeta_i^{\prime}}(y)  \varepsilon^2) + O(\varepsilon^5)
\\
&=
- 20\pi \alpha_3 \varepsilon^{\frac{1}{2}}
w_{\mu_i^{\prime},\zeta_i^{\prime}}(y)
|\sigma_1|^{\frac{1}{2}} \Bigl( \sigma_1 v_1 \bar\Lambda_1 + \sum_{l=2}^k \bar v_l \bar  \Lambda_l \Bigr)
+ O( w_{\mu_i^{\prime},\zeta_i^{\prime}}(y)^3  \varepsilon^2) + O(\varepsilon^5)  .
\end{align*}
The $O(\cdot )$ terms are bounded together with their derivatives with respect to $\zeta'$, $\mu'$.
Differentiating $E$ with respect to $\zeta'$ and $\bar\Lambda_l$, taking into account the last expression,  and thanks to
\eqref{cotaZetaGorro}, \eqref{cotaLambdaGorro}, \eqref{cotasSigma} and \eqref{cotaGradSigma1},
we find that  for $|y-\zeta_i|\leq \frac{\delta}{\varepsilon}$ the following hold
\begin{align*}
D_{\zeta'} E
&= O(\varepsilon^{2-\sigma} )w_{\mu_i^{\prime},\zeta_i^{\prime}}(y)^4
+O( w_{\mu_i^{\prime},\zeta_i^{\prime}}(y)^3  \varepsilon^2) + O(\varepsilon^5),
\end{align*}
\begin{align*}
D_{\bar\Lambda_1} E
&= O(\varepsilon^{ 2-\sigma } )w_{\mu_i^{\prime},\zeta_i^{\prime}}(y)^4
+O( w_{\mu_i^{\prime},\zeta_i^{\prime}}(y)^3  \varepsilon^2) + O(\varepsilon^5),
\end{align*}
and for $l=2,\dots,k$
\begin{align*}
D_{\bar\Lambda_l} E
&= O(\varepsilon )w_{\mu_i^{\prime},\zeta_i^{\prime}}(y)^4
+O( w_{\mu_i^{\prime},\zeta_i^{\prime}}(y)^3  \varepsilon^2) + O(\varepsilon^5)   .
\end{align*}
From this and analogous estimates outside of all the balls $B_{\delta/\varepsilon}(\zeta_i')$  it follows that
\begin{align*}
\|D_{\zeta'} \phi\|_* \leq \varepsilon^{ 2-\sigma } , \quad
\|D_{\bar\Lambda_1} \phi\|_* \leq \varepsilon^{ 2-\sigma}, \quad
\|D_{\bar\Lambda_l} \phi\|_* \leq \varepsilon, \quad l=2,\ldots,k.
\end{align*}
As a consequence,
\begin{align}
\label{theta2}
|D_\zeta\theta_\lambda^{(2)}(\zeta,\mu)   |\leq C \varepsilon^{3-\sigma},
\quad
|D_{\bar\Lambda_1} \theta_\lambda^{(2)}(\zeta,\mu) |\leq C \varepsilon^{4-2\sigma},
\quad
|D_{\bar\Lambda_l}\theta_\lambda^{(2)}(\zeta,\mu) |\leq C \varepsilon^{3-\sigma} ,
\end{align}
for $l=2,\ldots,k$.
(Here we are assuming $\sigma>0$ small so that $3-\sigma < 4 - 2 \sigma$).

Combining \eqref{theta1a}, \eqref{theta1b} and \eqref{theta2} we obtain the estimates
\eqref{cota1}, \eqref{cota2}, \eqref{cota3}.
\end{proof}

%
%
%

\section{The case of the annulus}
\label{exampleAnnulus}

Let $0<a<1$ and
\[
\Omega_a  = \{ x\in \R^3 \ : \  a < |x| < 1 \} .
\]
We want to show in this section that solutions with an arbitrary number of peaks exist for certain ranges of the parameter $a$.

\begin{prop}
\label{propA1}
Let $k\geq 2$ be fixed. Then there exists $a_k \in (0,1)$ such that for $a\in (a_k,1)$ there is $\lambda>0$ and a solution of $(\wp_\lambda)$ with $k$ concentration points.
\end{prop}

Explicit values of $a_k$ seem difficult to get, but one can obtain estimates that show that for a low number of peaks the annulus does not need to be so thin. In particular for two bubbles we have the following estimate.

\begin{prop}
\label{propA2}
For $a\in (\frac{1}{49},1)$ there is $\lambda>0$ and a solution of $(\wp_\lambda)$ with $2$ concentration points.
\end{prop}

Let us give first a lemma about the behavior of the Green function for a thin domain.
For this we write now  $G_0^a(x,y)$, $H_0^a(x,y)$, $g_0^a(x) = H_0^a(x,x)$ for the Green function, its regular part and the Robin function  respectively for $\lambda=0$ in the domain $\Omega_a$.
\begin{lemma}
\label{lemma81}
Let $x_0, y_0$ be fixed  so that $|x_0|=|y_0|=1$ and $y_0 \not= x_0$. Then
\begin{align}
\label{convergenceGreen}
G_0^a(y,x) \to 0
\end{align}
as $a \to 1$ uniformly for $y = r y_0$ with $r\in (a,1)$ and $x = r' x_0$ with $r'\in (a,1)$.
Moreover,
\begin{align}
\label{convergenceRobin}
\min_{\Omega_a} g_0^a \to \infty
\end{align}
as $a \to 0$.
\end{lemma}
\begin{proof}
To prove \eqref{convergenceGreen} let us write $\varepsilon= 1 -a>0$ and let $\varepsilon\to 0$. We also change the notation $G_0^a$ to $G_0^\varepsilon$, $\Omega_a$ to $\Omega_\varepsilon$ and shift coordinates so that the annulus is centered at $-e_1$:
\[
\Omega_\varepsilon = \{ z\in \R^3 : 1-\varepsilon<|z + e_1|<1\},
\]
where $e_1=(1,0,0)$.
Without loss of generality we can assume that $y_0 = 0$.
Our assumption now is that $|x_0+e_1|=1$ and $x_0\not=0$.

By the maximum principle
\[
0\leq G_0^\varepsilon(y,x) \leq \Gamma(y-x) , \quad \forall y\in \Omega_\varepsilon\setminus\{x\},
\]
for any $x\in \Omega_\varepsilon$.
Let  $\rho  =\frac{|x_0-y_0|}{4} >0$. Then there is $C$ such that
\[
0\leq G_0^\varepsilon(y,x) \leq C , \quad \forall y\in \Omega_\varepsilon \cap B_\rho(0),
\]
for any $x$ in the segment $\{ t x_0 + (1-t)(-e_1) : t\in (a,1)\}$.

Let
\[
\tilde G^\varepsilon(y') =  G_0^\varepsilon(\varepsilon y',x)
\quad  y' \in \frac{1}{\varepsilon}( \Omega_\varepsilon \cap B_\rho(0)) .
\]
Then $\tilde G^\varepsilon$ is harmonic and bounded in $\frac{1}{\varepsilon}( \Omega_\varepsilon \cap B_\rho(0)) $. By standard elliptic estimates, up to a subsequence,  $\tilde G^\varepsilon \to \tilde G$ which is harmonic and bounded on the slab $S = \{ (x_1,x_2,x_3) : -1<x_1<0 \}$, and vanishes on the boundary of this slab. We can then extend $\tilde G$ by reflections to a bounded harmonic function in $\R^3$. By the Liouville theorem $\tilde G$ is constant but then $\tilde G\equiv 0$.  Because the limit is unique we have the convergence for all $\varepsilon\to 0$, that is,
$\tilde G^\varepsilon (y')\to 0$ uniformly on compact subsets of $\frac{1}{\varepsilon}(\Omega_\varepsilon\cap B_\rho(0))$.
Therefore  $\tilde G^\varepsilon (y') \to  0$ uniformly
for $y' \in \{ (y_1',0,0) : y_1' \in [-1,0] \}$.
Changing variables back we obtain  \eqref{convergenceGreen}.

We now prove \eqref{convergenceRobin}.
We will use the maximum principle to compare the Green function of $\Omega_a$ with the Green function of suitable domains. First, Let $G_{B_1}$ denote the Green function of the unit ball $B_1 = B_1(0)$:
\[
-\Delta_y G_{B_1}(y,x) = \delta_x \quad \text{in }B_1, \quad G_{B_1}(y,x) = 0 \quad y \in \partial B_1.
\]
If $x\in \Omega_a$, the maximum principle guarantees that $G_0^a(y,x) \leq G_{B_1}$ in $\Omega_a$.
This implies that $g_0^a(x) \geq g_{B_1}(x)$ where  $g_{B_1}(x)$ denotes the Robin function in $B_1$. It is well known that $ g_{B_1}(x) \geq c \,\textrm{dist}(x,\partial B_1)^{-1}$ for some $c>0$. This implies that $\min_{\Omega_a} g_0^a \geq \frac{c}{1-a}\to \infty$ as $a\to 1$.
\end{proof}

To prove Propositions~\ref{propA1} and \ref{propA2} we consider a configuration of points in the $xy$ plane at equal distance from the origin and spaced at uniform angles, that is,
\[
\zeta_j(r) = ( r e^{2 \pi i \frac{j-1}{k}} , 0 ) \in \R^3  , \quad j=1,\ldots, k,
\]
where the notation we are using for $z\in \C$ and $t\in \R$, is $(z,t) = (Re(z),Im(z),t)$.

Define then the matrix $M_\lambda$ restricted to this configuration as
\[
\tilde M_\lambda(r) = M_\lambda(\zeta(r) ),
\]
where $\zeta(r) = (\zeta_1(r) , \ldots, \zeta_k(r) )$.
Similarly we define
\[
\tilde \psi_\lambda(r) = \tilde \psi_\lambda( \zeta(r))  ,
\]
and denote by $\tilde \sigma_j(\lambda,r)$  the eigenvalues of $\tilde M_\lambda(r)$ with $\tilde \sigma_1$ the smallest one.


\begin{proof}[Proof of Proposition~\ref{propA1}]
Let $k\geq 2$ be given.
By Lemma~\ref{lemma81}, if $a>0$ is small,  we have
\begin{align}
\label{positiveSigma1Lambda0}
& \tilde \sigma_j(0,r)  >0 , \quad \forall r\in (a,1), \ j=1,\ldots, k,
\\
\label{positivePairLambda0}
& g_0(\zeta_1(r))^2 - G_0(\zeta_1(r), \zeta_j(r))^2>0\quad \forall r\in (a,1), \ j=2,\ldots, k.
\end{align}
Now, we
define
\begin{align}
\label{lambda0}
\lambda_0 = \sup\, \{\lambda \in (0,\lambda_1): \sigma_j(\lambda',r) >0  \quad \forall r\in (a,1), \ j=1,\ldots, k , \ \lambda'\in (0,\lambda)	 \}.
\end{align}
Then $\lambda_0$ is well defined by continuity and \eqref{positiveSigma1Lambda0}. We will need the following properties:
\begin{align}
\label{positivePair}
& g_\lambda(\zeta_1(r))^2 - G_\lambda(\zeta_1(r), \zeta_j(r))^2>0\quad \forall
\lambda \in [0,\lambda_0),
r\in (a,1), \ j=2,\ldots, k,
\\
\label{lambdaLess}
& \lambda_0 <\lambda_1,
\\
\label{sigma1}
& \tilde \sigma_1(\lambda_0,r) \geq 0 \quad \text{and there exists $r_0 \in (a,1)$ such that $\sigma_1(\lambda_0,r_0)=0$},
\\
\label{simpleEigenvalue}
& \tilde \sigma_j(\lambda_0,r) >0 \quad \text{for all $r\in (a,1)$ and $ j=2,\ldots, k$},
\\
\label{sigma1Decreasing}
& \frac{ \partial \tilde \sigma_1 }{\partial \lambda}(\lambda_0,r)<0, \quad \forall r\in (a,1).
\end{align}
Let us prove  \eqref{positivePair}. If this fails, then for some $\lambda \in [0,\lambda_0)$, some $r_0 \in (a,1)$, and some  $ j=2,\ldots, k  $, we have $g_\lambda(\zeta_1(r))^2 - G_\lambda(\zeta_1(r), \zeta_j(r))^2 \leq 0$.
This condition implies that the matrix $\tilde M_\lambda(r)$ has a nonpositive eigenvalue. This follows from the criterion that asserts that a symmetric matrix $A=(a_{i,j})_{1\leq i,j\leq k}$ is positive definite if and only if all submatrices  $(a_{i,j})_{1\leq i,j\leq m}$ are positive definite for $ m=1,\ldots, k$ (we apply this to $\tilde M_\lambda(r)$ after the permutation of the rows 2 and $j$, and the columns 2 and $j$).
But this contradicts the definition of $\lambda_0$ \eqref{lambda0}.

Let us prove \eqref{lambdaLess}.  For this we recall that $\min_{\Omega_a} g_0>0$ and $\min_{\Omega_a} g_\lambda \to -\infty$ as $\lambda\uparrow\lambda_1$. Therefore there exists $r \in (a,1) $ and $\lambda \in (0,\lambda_1)$ such that $g_\lambda(\zeta_1(r)) = 0$.
This implies that $ g_\lambda(\zeta_1(r))^2 - G_\lambda(\zeta_1(r), \zeta_j(r))^2<0$
for any $ j=2,\ldots, k$.
By \eqref{positivePair} this value of $\lambda$ is greater or equal than $\lambda_0$. It follows that $\lambda_0<\lambda_1$.

Since $\lambda_0<\lambda_1$ by continuity we deduce the validity of \eqref{sigma1}.
We also deduce from this and the way we have arranged the eigenvalues that $\sigma_j(\lambda_0,r)\geq 0$ for all $ j=2,\ldots, k$ and for all $r\in (a,1)$.

To continue the proof of the stated properties we need a formula for the eigenvalues of a circulant matrix. We recall that a matrix $A$ of $k\times k$ is circulant if it has the form
\[
A =
\left[
\begin{matrix}
a_0 & a_{k-1} & a_{k-2} & \ldots & a_{2} & a_{1}
\\
a_1 & a_0 & a_{k-1} &  \ldots & a_{3} & a_{2}
\\
a_2 & a_1 & a_0 &  \ldots & a_{4} & a_{3}
\\
\vdots & \vdots& \vdots& &\vdots & \vdots
\\
a_{k-1} & a_{k-2} & a_{k-3} & \ldots & a_{1} & a_{0}
\end{matrix}
\right]
\]
for some complex numbers $a_0,\ldots,a_{k-1}$.
(This means each column is obtained from the previous one by a rotation in the components).
We note that the matrix $\tilde M_\lambda(r)$ has this structure with
\begin{align*}
a_0 &= g_\lambda(\zeta_1(r)),
\\
a_j &= -G_\lambda( \zeta_1(r), \zeta_{j+1}(r)), \quad   j=1,\ldots, k-1 ,
\end{align*}
since $G_\lambda (\zeta_l(r), \zeta_j(r) ) = G_\lambda (\zeta_{l+1}(r), \zeta_{j+1}(r) ) $.

It is known that the eigenvalues $\nu_l$ ($l=0,\ldots, k-1$) of the circulant matrix $A$ are given by
\[
\nu_l = \sum_{j=0}^{k-1} a_j e^{\frac{2\pi i}{k} j l } , \quad  l=0,\ldots, k-1.
\]
These numbers coincide up to  relabeling the indices with the numbers $\tilde\sigma_j(\lambda,r)$.
We note that since $\tilde M_\lambda(r)$ is symmetric, the eigenvalues are real.
We claim that
\[
\nu_0 < \nu_j \quad  j=2,\ldots, k-1.
\]
Indeed, since the $\nu_l$ are real
\begin{align*}
\nu_l & = g_\lambda(\zeta_1(r)) - \sum_{j=1}^{k-1}  Re\left[ G_\lambda(\zeta_1(r), \zeta_{j+1}(r))   e^{\frac{2\pi i}{k} j l }\right]
\\
& >  g_\lambda(\zeta_1(r)) - \sum_{j=1}^{k-1}  G_\lambda(\zeta_1(r), \zeta_{j+1}(r))    = \nu_0,
\end{align*}
where the strict inequality holds because there are point  $e^{\frac{2\pi i}{k} j l }$ in the sum which are not colinear and $G_\lambda(\zeta_1(r), \zeta_{j+1}(r)) >0$.
This proves \eqref{simpleEigenvalue} and also that
\begin{align}
\label{formulaSigma1}
\tilde \sigma_1(\lambda,r) = g_\lambda(\zeta_1(r)) - \sum_{j=1}^{k-1}  G_\lambda(\zeta_1(r), \zeta_{j+1}(r))  ,
\end{align}
for all $\lambda \in [0,\lambda_0]$ because for this range of $\lambda$ we know that the eigenvalues $\tilde \sigma_j$ are nonnegative.
From this formula we obtain
\begin{align*}
\frac{\partial  \tilde \sigma_1}{\partial \lambda}(\lambda,r)
&=
\frac{\partial  g_\lambda}{\partial \lambda}
(\zeta_1(r)) - \sum_{j=1}^{k-1}
\frac{\partial  G_\lambda }{\partial \lambda} (\zeta_1(r), \zeta_{j+1}(r)) < 0
\end{align*}
for $\lambda \in [0,\lambda_0]$, which proves  \eqref{sigma1Decreasing}.

Let us see that we are almost in a situation where  Theorem~\ref{thm1} can be applied.
Let $r_0$ be the number found in property \eqref{sigma1}.
The eigenvalue $\tilde \sigma_1(\lambda_0,r_0)$ is zero and $\tilde M_{\lambda_0}(r_0)$ is positive semidefinite (assumption (i)), we have $D_\zeta \sigma_1(\lambda_1,\zeta(r_0))=0$ because $\zeta(r_0)$ is a global minimum for $\sigma_1(\lambda_0,\cdot)$. Condition (iv) follows from \eqref{sigma1Decreasing}.

The only hypothesis in Theorem~\ref{thm1} which has not been verified is the nondegeneracy of $\zeta(r_0)$ as a critical point of $\sigma_1(\lambda_0,\cdot)$. In fact this nondegeneracy does not hold because the problem is invariant about rotations about the $z$ (or $x_3$) axis.
We could impose a symmetry condition on the functions involved so that degeneracy by rotation is eliminated, but still we do not know whether we have nondegeneracy in the radial direction.
Instead of this assumption, we will see that a slight modification of the argument in the proof of Theorem~\ref{thm1}  yields the desired conclusion. Basically,  the nature of the critical point of $F_\lambda$ in this case is stable with respect to $C^1$ perturbations.

We recall from Section~\ref{secReduction} that to construct a solution it is sufficient to find a critical point of the function
$
\bar J_\lambda( \sum_{j=1}^k V_j + \phi)
$
and
\[
\bar J_\lambda\Bigl(\sum_{j=1}^k V_j + \phi\Bigr)
= J_\lambda\Bigl(\sum_{j=1}^k U_j \Bigr ) + o(\varepsilon^{2})
\]
where $o(\varepsilon^{2})$ is in $C^1$ norm.
Therefore  it is enough to ensure that $J_\lambda(\sum_{j=1}^k U_j ) $ has a critical point that is stable under $C^1$ perturbations.

In the case when $\Omega_a$ is an annulus, and $\zeta_j(r) = (re^{2\pi i\frac{j-1}{k}},0)$ using that $g_\lambda(\zeta_j(t))$ only depends on $r$ and considering $\mu = \mu_1 = \ldots = \mu_k $, by Lemma~\ref{lemmaEnergyExpansion} we have that
\[
J_\lambda\Bigl(\sum_{j=1}^k U_j  \Bigr) = F_\lambda(\mu,r) + R_\lambda(\mu,r),
\]
where
\[
F_\lambda(\mu,r) = k a_0 + 2 a_1 \mu f_\lambda(r) + k a_2 \lambda \mu^2 - a_3 \mu^2 f_\lambda(r)^2
\]
with
\begin{align*}
f_\lambda(r) &= k g_\lambda(\zeta_1(r)) - k \sum_{j=2}^k G_\lambda(\zeta_1(r), \zeta_j(r))
\end{align*}
and
\[
R_\lambda(\mu,r) = O(\mu^{3-\sigma}).
\]
for some $\sigma\in(0,1)$.

As was observed previously, for $\lambda \in [0,\lambda_0]$, $f_\lambda(r)$ is precisely the eigenvalue $\tilde \sigma_1(\lambda,r)$  (see  \eqref{formulaSigma1}).
Therefore \eqref{sigma1} gives $f_{\lambda_0}(r)\geq 0$ and then there exists $r_0\in (a,1)$ such that $f_{\lambda_0}(r_0)=0$.

Since we have \eqref{sigma1Decreasing} we deduce that for $\lambda = \lambda_0 +\varepsilon$ with $\varepsilon>0$ small enough and $r$ close to $r_0$,
we have $f_\lambda(r)<0$ and so the equation
\[
\frac{\partial}{\partial \mu}
F_\lambda(\mu,\zeta) = 0
\]
has a solution given explicitly by
\[
\mu_0(\lambda,r) = \frac{- a_1  f_\lambda(r)}{k a_2 \lambda - a_3 f_\lambda(r)^2} > 0.
\]
We consider this expression only for $r$ in a neighborhood of $r_0$, so that  $f_\lambda(r) \leq  0$.
Then
\[
\frac{\partial}{\partial \mu} ( F_\lambda(\mu,r) +R_\lambda(\mu,r) ) = 0
\]
has a solution $\mu(\lambda,r) $ close to $\mu_0(\lambda,r) $.
Note that since $\frac{\partial}{\partial \mu} R_\lambda(\mu,r) = O(\mu^{2-\sigma})$, we have
\[
|\mu(\lambda,r)- \mu_0(\lambda,r)|\leq C | f_\lambda(r)|^{2-\sigma}.
\]
Replacing $\mu(\lambda,r)$ in $F_\lambda$ we find
\[
F_\lambda( \mu(\lambda,r), r )  + R_\lambda( \mu(\lambda,r), r ) =
-\frac{a_1^2 f_\lambda(r)^2}{k a_2 \lambda-a_3 f_\lambda(r)} + O( | f_\lambda(r)|^{3-\sigma} ).
\]
From this formula, \eqref{sigma1Decreasing} and the property
\[
f_\lambda(r) \to \infty \quad \text{ as \, $r\to a$ or $r\to 1$},
\]
we get that $F_\lambda( \mu(\lambda,r), r )  + R_\lambda( \mu(\lambda,r), r )  $ has a critical point $r_\lambda$ for which $f_\lambda(r_\lambda)<0$.
\end{proof}

\begin{proof}[Proof of Proposition~\ref{propA2}]
The argument is the same as in Proposition~\ref{propA1}, except that for this result we claim that properties \eqref{positiveSigma1Lambda0} and \eqref{positivePairLambda0} hold for $a\in (\frac{1}{49},1)$.
In the case $k=2$ both properties actually follow from the following claim: if $a \in (\frac{1}{49},1)$ then
\begin{align}
\label{g1}
g_0(x) > G_0(x,-x) , \quad \forall x\in \Omega_a.
\end{align}
To prove this we  use an explicit formula for the Green function in the annulus $ \Omega_a$, which can be found in \cite{grossi-vujadinoic}, to obtain that:
\[
g_0(x) = \frac{1}{\omega_{2}}
\sum_{m=0}^\infty
P_m(x)
\quad \text{and}\quad
G_0(x,-x) = \frac{1}{\omega_{2}}\left[
\frac{1}{2\vert x\vert}-
\sum_{m=0}^\infty
(-1)^m P_m(x)\right],
\]
where
\[
P_m(x):=\frac{a^{2m+1}-2a^{2m+1} |x|^{2m+1} +|x|^{2(2m+1)} }{(2m+1) |x|^{2(m+1)}(1-a^{2m+1}) } .
\]
Notice that $P_m(x)$ is nonnegative for all $m\geq 0$, and therefore,
\begin{align*}
g_0(x) - G_0(x,-x) &= \frac{1}{\omega_{2}}
\left[-
\frac{1}{2\vert x\vert}+
\sum_{m=0}^\infty
[1+(-1)^m] \,P_m(x)\right]\\
&\geq
\frac{1}{\omega_{2}}
\left[-
\frac{1}{2\vert x\vert}+
2P_0(x)\right]\qquad \forall x\in \Omega_a.
\end{align*}
A sufficient condition to have \eqref{g1} is then
\[
4 \frac{a-2a|x|+|x|^2}{|x|^2(1-a)}> \frac{1}{|x|}, \quad \forall x\in\Omega_a.
\]
This in turn holds if $a\in(\frac{1}{49},1)$.
\end{proof}

\bigskip \noindent \textbf{Acknowledgement.}
The research of M. Musso has been partly supported by FONDECYT Grant 1160135 and Millennium Nucleus Center for Analysis of PDE, NC130017.
D. Salazar was partially funded by grant Hermes 35454 from Universidad Nacional de Colombia sede Medell\'\i n and
Millennium Nucleus Center for Analysis of PDE, NC130017.

\end{document}